\documentclass[11pt,reqno]{amsart}
\usepackage{lmodern}
\usepackage[latin1]{inputenc}
\usepackage{amsfonts}
\usepackage{amssymb}
\usepackage{amsmath}
\usepackage{amsthm}
\usepackage{cite}
\usepackage{graphicx}
\usepackage{amscd}
\usepackage{color}
\usepackage{bm}
\usepackage{enumerate}
\usepackage[dvips]{epsfig}
\usepackage{psfrag}
\usepackage{epsfig}
\usepackage{comment}
\usepackage{hyperref} 
\usepackage{orcidlink}

\usepackage[abbrev,backrefs]{amsrefs}
\allowdisplaybreaks

\usepackage{cite}
\newtheorem{theorem}{Theorem}
\newtheorem{definition}[theorem]{Definition}
\newtheorem{corollary}[theorem]{Corollary}
\newtheorem{proposition}[theorem]{Proposition}
\newtheorem{lemma}[theorem]{Lemma}
\newtheorem{remark}[theorem]{Remark}

\numberwithin{theorem}{section}

\numberwithin{equation}{section}

\newcommand{\abs}[1]{\lvert#1\rvert}

\newcommand{\R}{\mathbb{R}}

\newcommand{\N}{\mathbb{N}}

\newcommand{\cB}{{\mathcal B}}

\newcommand{\cM}{{\mathcal M}}

\newcommand{\weak}{\rightharpoonup}

\newcommand{\B}{{\bf B}}

\newcommand{\eps}{\varepsilon}

\renewcommand{\phi}{\varphi}

\renewcommand{\epsilon}{\varepsilon}

\let\psfragfont\footnotesize
\let\psfragfonta\tiny

\psfrag{Oe}{\psfragfont$\Omega_\eps$}
\psfrag{z0}{\psfragfont$\tilde z_0$}
\psfrag{Ke}{\psfragfont$K_\eps$}
\psfrag{A}{\psfragfont$A_\eps$}
\psfrag{twoeps}{\psfragfont$2\eps$}
\psfrag{0}{\psfragfont$0$}
\psfrag{Y0}{\psfragfont$Y_0$}
\psfrag{rn-1}{\psfragfont$\R^{N-1}$}
\psfrag{rn+}{\psfragfont$\R^N_+$}
\psfrag{or}{\psfragfonta$0$}
\psfrag{u}{\psfragfont$u$}
\psfrag{u?}{\psfragfont$u$ ?}
\psfrag{e1}{\psfragfonta$e_1$}
\psfrag{Al}{\psfragfonta$A_\lambda?$}
\psfrag{-e1}{\psfragfonta$-e_N$}
\psfrag{B}{\psfragfont$\B$}
\psfrag{k}{\psfragfonta$K$}
\psfrag{eh}{\psfragfonta$e_H$}
\psfrag{alh}{\psfragfonta$\alpha_H$}
\psfrag{sigma}{\psfragfonta$\Sigma$}
\psfrag{om}{\psfragfonta$\Omega$}
\psfrag{om1}{\psfragfonta$\Omega_1$}
\psfrag{om2}{\psfragfonta$\Omega_2$}
\psfrag{dh}{\psfragfonta$T_\lambda$}
\psfrag{ph}{\psfragfonta$T_\lambda$}
\psfrag{h}{\psfragfonta$K_\lambda$}
\psfrag{x}{\psfragfont$x$}
\psfrag{z}{\psfragfont$z$}
\psfrag{xl}{\psfragfont$x^\lambda$}
\psfrag{xk}{\psfragfonta$y_k$}
\psfrag{ex}{\psfragfont$e$}
\psfrag{eh}{\psfragfont$e_H$}
\psfrag{y}{\psfragfont$y$}
\psfrag{o}{\psfragfont$0$}
\psfrag{x1}{\psfragfont$\overline x$}
\psfrag{x_1}{\psfragfont$x_1$}
\psfrag{y1}{\psfragfont$\overline y$}
\psfrag{yl}{\psfragfont$y^\lambda$}
\psfrag{lam}{\psfragfont$ $}
\psfrag{b}{\psfragfont$\B$}
\psfrag{pb}{\psfragfont$\partial \B$}
\psfrag{sr}{\psfragfont$S_r$}

\author[A. Cannone]{Alessandro Cannone \orcidlink{https://orcid.org/0009-0006-4073-8298} }
\address{\noindent Alessandro Cannone \newline
	Dipartimento di Matematica, Universit\`{a} degli Studi di Bari Aldo Moro,\newline
	Via Orabona 4, 70125 Bari, Italy.}\email{alessandro.cannone@uniba.it}

\author[M. Yu]{Meng Yu \orcidlink{https://orcid.org/0009-0003-8065-4863}}
\address{\noindent Meng Yu\newline Institut f\"ur Mathematik, Goethe-Universit\"at Frankfurt\newline Frankfurt am Main 60054, Germany}
\email{  yumeng161@mails.ucas.ac.cn}

\begin{document}

	\title[Extremal functions for nonlocal interaction functionals in dimension $N$]{Blow-up analysis and extremal functions for nonlocal interaction functionals in dimension $N$}

\begin{abstract}
    In this paper we study Moser-Trudinger type inequalities for some nonlocal energy functionals in  presence of a logarithmic convolution potential, when the domain is a ball of $\R^N$ with $N \geq 2$. In particular, we perform a blow-up analysis to prove existence of extremal functions in the borderline case of critical growth. Using this, we extend the results in \cite{CiWeYu} to higher dimension and sharpen \cite{CC}.
\end{abstract}
\maketitle
\section{Introduction}
In the last century, the study of nonlocal interaction equations has attracted great interest. These kinds of problems have emerged from various applications in different contexts, ranging from vortex theory,
statistical dynamics of selfgravitating clouds \cite{Suzuki,W},  quantum theory for crystals 
\cite{Dolbeault-Perthame},  to the description of vortices in turbulent Euler flows \cite{CLMP}.

Recently, in \cite{CiWe2}, the authors established Trudinger-Moser inequalities in the presence of a logarithmic kernel in the planar case, characterizing the critical non-linear growth rates for these inequalities. Then in \cite{CC} it was generalized for any dimension $N\geq2$.\\ \noindent
The problem in the general case is to maximize the quantity\begin{equation}
    \label{maximization problem}
    m_1(N,F):=\sup\limits_{u\in \cB_1} \Phi(u),
\end{equation}
where\begin{equation}\label{nonlocal functional}
    \Phi(u):=\int_{B_1} \int_{B_1} \ln{\frac{1}{|x-y|}} F(u(x))F(u(y))\, dx\, dy,
\end{equation} is the nonlocal interaction functional and 
\begin{equation}\label{set of function}
    \cB_1:=\{u \in W^{1,N}_0(B_1)\, |\, |\nabla u|_N \leq 1\}
\end{equation} is the set of functions and $|\cdot|_N$ is the standard Lebesgue norm.\\
On the nonlinearity $F$ we have the following condition: 
\begin{itemize}
\item[$(F_0)$] $F: \R \to [0,\infty)$ is even and continuous  on $[0,\infty)$. Moreover, there exist constants $\alpha,c>0$ with 
	\begin{equation}
		\label{eq:general-growth-condition}
		F(t) \le c e^{\alpha |t|^{N/(N-1)}}  \qquad \text{for $t \in \R$},
	\end{equation}
\end{itemize}which ensures that the double integral in \eqref{nonlocal functional} is well-defined for functions $u \in \mathrm{W}^{1,N}_0(B_1)$. Indeed, splitting the kernel $\ln \frac{1}{| \cdot |}$ into its positive and negative part and defining the functionals $\Phi^{\pm} \colon \mathcal{M}(B_1) \to [0, \infty]$ by
\begin{equation*}
    \Phi^{\pm}(u) = \int_{B_1} \int_{B_1} \ln^{\pm} \frac{1}{|x - y|} F(u(x))F(u(y)) \, \text{d}x \, \text{d}y,
\end{equation*}
where $\ln^{\pm} = \max\{\pm \ln, 0\}$ and $\mathcal{M}(B_1)$ denotes the space of the real Lebesgue-measurable functions on $B_1$, it follows from \cite{CC} that $\Phi^{\pm}(u) < \infty$ for every $u \in \mathrm{W}^{1,N}_0(B_1)$, and therefore the quantity in \eqref{nonlocal functional} has a well-defined finite value
\begin{equation*}
    \Phi(u) := \int_{B_1} \int_{B_1} \ln \frac{1}{|x - y|} F(u(x))F(u(y)) \, \text{d}x \, \text{d}y = \Phi^{+}(u) - \Phi^{-}(u) \quad \text{for every } u \in \mathrm{W}^{1,N}_0(B_1).
\end{equation*}
In order to study the maximization problem  we stress that  the value 
 $\alpha_N:= N \omega_{N-1}^{1/(N-1)}$, where $\omega_{N-1}$ is the measure of the surface of the unit ball of $\R^N$, plays a key role.
We stress that $\omega_{N-1}$ tends to zero as the dimension goes to $+ \infty$.
\\ \noindent We begin by remarking that if  $F$ satisfies $(F_0)$ with $\alpha < \alpha_N$,  then 
\begin{equation}\label{rottt}
	m_1(F,N) := \sup_{u \in \cB_1} \Phi(u) <\infty
\end{equation}
by the classical Trudinger-Moser inequality  \cites{trudinger,moser} and the logarithmic Hardy-Littlewood-Sobolev inequality \cite[Theorem 2]{beckner}, as  shown in \cite{CC}.\\ \noindent Regarding the critical case $\alpha_N$, in \cite{CC}, the following theorem was proven under an additional condition on $F$, namely that it is increasing on $[0, +\infty)$.

\begin{theorem}\cite[Theorem 1.2]{CC}
	\label{sec:introduction-main-thm-C-F2}
	Suppose that $F$ satisfies $(F_0)$ and is increasing on $[0,+\infty)$.
\begin{enumerate}
		\item If $F$ has at most $\beta$-critical growth for some $\beta \le -\frac{N}{2(N-1)}$, then 
$$
m_1(N, F) := \sup_{u \in \cB_1} \Phi(u) <\infty.
$$
		\item If $F$ has at most $\beta$-critical growth for some $\beta < -\frac{N}{2(N-1)}$, then $m_1(N,F)$ is attained, and every maximizer for $\Phi$ in $\cB_1$ is, up to sign, a radial and radially decreasing function in $\cB_1$.\\  
\item If $F$ has at least $\beta$-critical growth for some $\beta>-\frac{N}{2(N-1)} $, then 
$m_1(N,F)= \infty$.
\end{enumerate}
\end{theorem}
Here, \textit{at most $\beta$-critical growth} means that $F(t)\leq c \, e^{\alpha_N|t|^{\frac{N}{N-1}}}(1+|t|)^{\beta}$ for $t \in \R$ with some constant $c >0$, while \textit{at least $\beta$-critical growth} means that there exist $t_0$, $c>0$ with the property that $F(t)\geq c\, e^{\alpha_N |t|^{\frac{N}{N-1}}}|t|^\beta$ for $|t|\geq t_0$.\\ \noindent
We note that the existence of extremal functions in the case of the $\beta$-threshold is an open problem, which will be addressed in this article. We aim to prove their existence through a blow-up analysis, as was done in the two-dimensional case for $\beta=-1$ in \cite{CiWeYu}.\\ \noindent
In our critical case $\beta=-\frac{N}{2(N-1)}$, we assume in the following that $F: \R \to [0,+\infty)$ has the form 
\begin{equation}
  \label{eq:F-shape}
F(t)= \frac{e^{\alpha_N |t|^{\frac{N}{N-1}}}}{ ( 1 + |t| )^\frac{N}{2(N-1)} } g(|t|),
\end{equation}
where
\begin{itemize}
\item[$(g_0)$]  $g: [0,\infty) \to [0,\infty)$ satisfies  \begin{equation} \label{eq:g-0-first} g(t) \le \gamma e^{\gamma |t|^{N/(N-1)}} \qquad \text{for $t \ge 0$ with some $\gamma \ge 1$.} 
\end{equation}
\end{itemize}
For our main results, we also need the following monotonicity condition.
\begin{itemize}
\item[$(g_1)$] The function $g$ is of class $C^1$ and satisfies $F' \ge 0$ on $(0,\infty)$, i.e.,
  \begin{equation}\label{strictlyincreasing}
  	\frac{N}{2(N-1)} (2 \alpha_N t^\frac{N}{N-1} + 2 \alpha_N t^\frac{1}{N-1} -1) g(t)+(1+t)g'(t)  \ge 0 \quad \forall t \in (0, \infty).
	\end{equation}
\end{itemize}
  
\begin{remark}
\label{remark-g-0-assumption}  
        We also note for later use that, if (\ref{strictlyincreasing}) holds, we have
        $$g'(t) \ge \frac{N}{2(N-1)} \frac{1- 2 \alpha_N t^\frac{1}{N-1} - 2 \alpha_N t^\frac{N}{N-1} }{1+t}g(t) \ge 0$$ and therefore $g(t) \ge g(0)$ for $t \in [0,t_0]$, where $t_0 > 0$ is the first point in $(0,+\infty)$ such that $ 1 - 2 \alpha_N t_0^\frac{1}{N-1} - 2 \alpha_N t_0^\frac{N}{N-1} = 0 $.
\end{remark}
In our first theorem, we provide sharp borderline conditions for the problem of maximizing $\Phi$ in $\cB_1$ depending on the asymptotic behaviour of the function $g$ at infinity. To state our main results, we define
\begin{equation}
  \label{eq:def-C-g}
C_g:= \limsup \limits_{t \to + \infty} {g(t)} \in  [0,\infty].  
\end{equation}

\begin{theorem}	\label{sec:introduction-main-thm-last}
	Suppose that $g$ satisfies $(g_0)$.
\begin{enumerate} 
\item If $C_g< \infty$, then $m_1(F) <\infty$.
\item If $\lim \limits_{t \to + \infty} g(t)=  + \infty$,
then $m_1(F)= + \infty$.
\item If $g$ also satisfies $(g_1)$ and $C_g = 0$, 
	 then $m_1(F)$ is attained, and every maximizer for $\Phi$ in $\cB_1$ is, up to sign, a radial and radially decreasing function in $\cB_1$.
\end{enumerate}
\end{theorem}
\begin{remark}
  \begin{itemize}
  \item[(i)] Theorem~\ref{sec:introduction-main-thm-last}(i) is essentially contained in \cite[Theorem 1.2]{CC}, since $F$ has at most $(-1)$-critical growth if $C_g< \infty$. However, as has been mentioned already, it was assumed in addition in \cite[Theorem 1.2]{CC} that $F$ is increasing on $[0,\infty)$, and we shall note in the proof of Proposition~
\ref{C-finiteness}, Section~\ref{sec:preliminaries}   below that this restriction is unnecessary.      
  \item[(ii)] 
We notice that Theorem \ref{sec:introduction-main-thm-last}(ii) and (iii) both improve \cite[Theorem 1.2]{CC}.
Indeed, consider first $\sigma \ge \frac{N}{N-1} \log 3$ and the function 
$$
t \mapsto g(t)=  \log^\sigma (2+ t), \quad t \ge 0
$$
Then $g$ satisfies $(g_0)$, $(g_1)$ and $\lim \limits_{t \to + \infty} g(t)=  + \infty$, so $m_1(F)= + \infty$ by Theorem~\ref{sec:introduction-main-thm-last}(ii). However, $F$ has not at least $\beta$-critical growth if $\beta >-\frac{N}{2(N-1)}$, so 
\cite[Theorem 1.2]{CC} does not apply.
Similarly, we may consider $0 < \sigma <\frac{N}{(N-1)}\log 2$ and the function 
$$
t \mapsto g(t)=  \frac{t}{2+t}\log^{-\sigma} (2+ t), \quad t \ge 0
$$
Then $g$ satisfies $(g_0)$, $(g_1)$ and $C_g = 0$, so $m_1(F)$ is attained by Theorem~\ref{sec:introduction-main-thm-last}(iii). On the other hand,  $F$ does not have at most $\beta$-critical growth if $\beta <-\frac{N}{2(N-1)}$, so \cite[Theorem 1.2]{CC} does not apply.
\end{itemize}
\end{remark}

Theorem~\ref{sec:introduction-main-thm-last} leaves open the question whether $m_1(F)$ is attained in the purely critical case where
$C_g \in (0, \infty)$. The study of this question is the context of the main results of this paper. Our first answer to this question is the following conditional result.

\begin{theorem}	\label{sec:introduction-main-thm-critical-sufficient-cond}
	Suppose that $g$ satisfies $(g_0)$, $(g_1)$ and $C_g \in (0, \infty)$. 
        If there exists $u \in \cB_1$ with
        \begin{equation}
          \label{eq:sufficient-cond}
        \Phi(u) > \frac{2 \omega_{N-1}^2}{N^3} \Bigl(\frac{g^2(0)}{4} + \frac{N}{2}  \omega_{N-1}^\frac{1}{N-1} C_g^2 e^{2(1+\frac{1}{2}+\cdots+ \frac{1}{N-1})}   \Bigr),          
        \end{equation}
then $m_1(F)$ is attained in $\cB_1$.
\end{theorem}

The sufficient condition (\ref{eq:sufficient-cond}) is a consequence of a detailed analysis of $\Phi(u_n)$ for concentrating sequences. For this we need the following definition.

\begin{definition}
  \label{def-scs-sequence}
We call a sequence of functions $u_n \in \cB_1$ a {\em Schwarz symmetric concentrating sequence} ($SCS$-sequence in short) if $u_n$ is radial, nonnegative and nonincreasing in the radial variable for every $n$ and satisfies $u_n \weak 0$ weakly in $W^{1,N}_0(B_1)$ but not strongly.  
\end{definition}

We then have the following upper bound.

\begin{theorem}	\label{sec:introduction-CC-theorem}
  Suppose that $g$ satisfies $(g_0)$, $(g_1)$ and $C_g \in (0, \infty)$. 
  \begin{itemize}
  \item[(i)] For any $SCS$-sequence $(u_n)_n \subset \cB_1$ we have
        \begin{equation}
          \label{limsup-SCS}
          \limsup_{n \to \infty} \Phi(u_n) \le \frac{2 \omega_{N-1}^2}{N^3} \Bigl(\frac{{g^2(0)}}{4} + \frac{N}{2}  \omega_{N-1}^\frac{1}{N-1} C_g^2 e^{2(1+\frac{1}{2}+\cdots+ \frac{1}{N-1})} \Bigr).
        \end{equation}
  \item[(ii)] If $\lim \limits_{t \to \infty}g(t)=C_g$, then there exists a $SCS$-sequence $(u_n)$ with
        \begin{equation}
          \label{sharp-lim-SCS}
     \lim_{n \to \infty} \Phi(u_n)= \frac{2 \omega_{N-1}^2}{N^3} \Bigl(\frac{{g^2(0)}}{4} +\frac{N}{2}  \omega_{N-1}^\frac{1}{N-1} C_g^2 e^{2(1+\frac{1}{2}+\cdots+ \frac{1}{N-1})}  \Bigr),
\end{equation}
so the upper bound in (\ref{limsup-SCS}) is sharp.  If, in addition,
\begin{equation}
\label{g-2-prime} \liminf \limits_{\tau \to \infty} \Bigl(g(\tau)-{C_g}\Bigr)\tau^\rho >0  \qquad \qquad \text{for some $\rho<\frac{1}{N-1}$,}
\end{equation}
then this sequence satisfies 
        \begin{equation}
          \label{sharp-lim-SCS-above}
\Phi(u_n)> \frac{2 \omega_{N-1}^2}{N^3} \Bigl(\frac{{g^2(0)}}{4} + \frac{N}{2}  \omega_{N-1}^\frac{1}{N-1} C_g^2 e^{2(1+\frac{1}{2}+\cdots+ \frac{1}{N-1})}   \Bigr) \qquad \text{for large $n$.}
\end{equation}
  \end{itemize}
\end{theorem}

Theorem~\ref{sec:introduction-main-thm-critical-sufficient-cond} and Theorem~\ref{sec:introduction-CC-theorem}(ii) immediately give rise to the following Theorem on the existence of maximizers. 

\begin{theorem}	\label{sec:introduction-main-thm-critical-asymptotic}
  Suppose that $g$ satisfies $(g_0)$, $(g_1)$, $\lim \limits_{t \to \infty}g(t)=C_g \in (0,\infty)$ and (\ref{g-2-prime}). Then $m_1(F)$ is attained in $\cB_1$.
\end{theorem}

This paper is structured as follows. Section \ref{sec:preliminaries} is dedicated to reviewing some preliminaries concerning nonlocal interaction energies and demonstrating the finiteness of the supremum $m_1(F)$, as established in Theorem~\ref{sec:introduction-main-thm-last}(i). In Section \ref{sec:unbounded-case}, we address the unbounded case and provide the proof of Theorem \ref{sec:introduction-main-thm-last}(ii). Section \ref{sec:cont-prop-phi} then deals with the existence of extremal functions in the subcritical case, thereby completing the proof of Theorem \ref{sec:introduction-main-thm-last}.

Subsequently, in Section \ref{sec:an-upper-bound}, we perform the change of variables motivated by the work of Carleson and Chang \cite{carlesonchang}, and determine an upper bound for the Schwarz symmetric concentrating sequences. Section \ref{sec:sharpn-upper-limit} is devoted to answering an open problem raised in \cite{CC} and proving the second part of Theorem \ref{sec:introduction-CC-theorem}. To this end, we first establish Lemma \ref{xi-choice}, a technical result that plays a crucial role in the proof of the theorem. The construction of an SCS-sequence satisfying \eqref{sharp-lim-SCS} draws inspiration from Figueiredo et al. \cite{figuereido-do-o-ruf}. However, it is necessary to operate in a complementary parameter regime, and the estimates involved differ significantly from those presented in \cite{figuereido-do-o-ruf}.

\medskip
{\bf Notation.} Throughout this paper, if $v: \R^N \to \R$ is a radially symmetric function, we let $v$ also denote the associated function $[0,\infty) \to \R$ of the radial variable $r=|x|$.

\bigskip
\noindent
{\bf Acknowledgments.} The first author is supported by the INdAM-GNAMPA Project (CUP: E53C25002010001) and by the PhD scholarship 
\textit{PDEs from Quantum Science} (CUP: H91I23000500007). \\ \noindent The authors would like to thank Prof. Silvia Cingolani and Prof. Tobias Weth for their valuable comments and inspiring discussions that improved the presentation of this paper.

\section{Preliminaries and the first part of Theorem \ref{sec:introduction-main-thm-last}}
\label{sec:preliminaries}

In this section we  will complete the proof of Theorem~\ref{sec:introduction-main-thm-last}(i), but before it is useful to introduce some notation and recall some preliminary results.\\ \noindent Let $\cM_+(\R^N)$ denote the cone of nonnegative real-valued measurable functions on $\R^N$. If $\Omega \subset \R^N$ is a measurable subset and $u$ is a nonnegative real-valued measurable function on $\Omega$, we also regard $u$ as a function in $\cM(\R^N)$ by trivial extension. We then define the quadratic forms $b_\pm: \cM_+(\R^N) \to [0,\infty]$ by 
\begin{equation*}
(v,w) \mapsto b_\pm(v,w)= \int_{\R^N} \int_{\R^N}
\ln^\pm \! \frac{1}{|x-y|}\, v(x)w(y)\,dx dy.
\end{equation*}
Moreover, we define 
$$
b_0(v,w):=b_+(v,w)-b_-(v,w) 
$$
for all functions $v,w \in \cM_+(\R^N)$ for which $b_\pm(v,w)<\infty$. For the sake of brevity, we also set
$$
b_\pm (v):= b_\pm(v,v) \qquad \text{and}\qquad b_0(v):= b_0(v,v)\quad \text{if $b_+(v) < \infty$.}
$$
By definition, we then have
$$
\Phi_\pm(u)= b_\pm(1_{B_1} F(u)) \qquad \text{and}\qquad   \Phi(u)= b_0(1_{B_1} F(u)) \qquad \text{for $u \in W^{1,N}_0(B_1)$.}
$$
Moreover, as noted in the introduction, all of these values are finite if $u \in W^{1,N}_0(B_1)$. We also recall from the introduction that 
 $$
 \cB_1:= \{u \in W^{1,N}_0(B_1)\::\: |\nabla u|_N \le 1\}.
 $$
In order to study the maximization problem for $\Phi$ in the set $\cB_1$, it is important to note that the functional $\Phi$ increases under Schwarz symmetrization if $F$ is an even, nonnegative and increasing on $[0,\infty)$. This is the consequence of the following Riesz rearrangement type inequalities noted in \cite[Lemma 2.3]{CiWe2}. For $v \in \cM_+(\R^N)$ we have
\begin{equation}
  \label{eq:Riesz-rearrangement}
b_+(v^*) \geq b_+(v)\qquad \text{and}\qquad b_-(v^*) \leq b_-(v), 
\end{equation}
where, here and in the following, $v^*$ denotes the Schwarz symmetrization of $v$. We then let 
$$
\cB_1^*:=  \{u^* \::\: u \in \cB_1\}
$$
denote the corresponding Schwarz symmetrized set of $\cB_1$. By the Polya-Szego inequality, we have $\cB_1^* \subset \cB_1$.
Hence the following key corollary readily follows from~\eqref{eq:Riesz-rearrangement}.

 \begin{corollary}
   \label{reduction-to-the-radial-case}
   Let $g$ satisfy $(g_0)$ and $(g_1)$, and let $F$ be given by (\ref{eq:F-shape}). Then we have
\begin{equation}
  \label{eq:m-1-symmetrization-identity}
   m_1(F)= \sup_{u \in \cB_1^*}\Phi(u)  \qquad \text{and}\qquad  m_1^+(F)= \sup_{u \in \cB_1^*}\Phi^+(u),
\end{equation}
where $m_1^+(F)= \sup \limits_{\cB_1}\Phi^+$.
\end{corollary}

 \begin{proof}
   Since $F$ is nonnegative, even and increasing on $[0,\infty)$ by assumptions $(g_0)$ and $(g_1)$, we have $[1_{B_1}F(u)]^*= 1_{B_1}F(u^*)$. Therefore
$$
\Phi(u^*)= b_0(1_{B_1} F(u^*))= b_0([1_{B_1}F(u)]^*) \ge b_0(1_{B_1}F(u)) = \Phi(u) 
$$
and
$$
\Phi^+(u^*)= b_+(1_{B_1} F(u^*))= b_+([1_{B_1}F(u)]^*) \ge b_+(1_{B_1}F(u)) = \Phi^+(u)  \qquad \text{for $u \in W^{1,N}_0(B_1)$.}
$$
Since moreover $u^* \in \cB_1^* \subset \cB_1$ for $u \in \cB_1$, the claim follows.
\end{proof}

As a consequence of Corollary~\ref{reduction-to-the-radial-case}, it is important to study the restrictions of the maps $b_\pm$ to radial functions. As noted in \cite[Corollary 2.8]{CiWe2}, for {\em radial} functions $v,w \in \cM_+(\R^N)$ with $b_\pm(v)< \infty$ we have, by Newton's theorem, 
\begin{equation}\label{newton-2-arguments}
\frac{b_0(v,w)}{\omega_{N-1}^2}= \int_{0}^1 r^{N-1} w(r) \Bigl(\ln \frac{1}{r}  \int_{0}^r \rho^{N-1}  v(\rho)d\rho + \int_{r}^1 \rho^{N-1} (\ln \frac{1}{\rho}) v(\rho)d\rho\Bigr)dr
 \end{equation}
and 
\begin{equation}
\label{newton-1-argument}
\frac{b_0(v)}{\omega_{N-1}^2}= 2\int_{0}^\infty r^{N-1} v(r) \ln \frac{1}{r} \int_{0}^r \rho^{N-1}  v(\rho)d\rho dr. 
\end{equation}

We also note the following lemmas.
\begin{lemma}
\label{general-boundedness} Let, for $i=1,2$, $g_i$ be $C^1$ nonnegative bounded even functions, and let $\cB_{1,rad}:= \{u \in \cB_1\::\: \text{$u$ radial}\}$. Then
\begin{equation}
  \label{eq:def-phi-beta-1-beta-2}
\Phi_{g_1,g_2}(u_1,u_2):= \int_{0}^1 r 
\frac{e^{\alpha_N u_1^{\frac{N}{N-1}}(r)} g_1(u_1(r))}{(1+ |u_1(r)|)^{\frac{N}{2(N-1)}}} 
\ln \frac{1}{r} \int_{0}^r \rho \frac{e^{\alpha_N u_2^{\frac{N}{N-1}}(\rho)} g_2(u_2(\rho))}{(1+ |u_2(\rho)|)^{\frac{N}{2(N-1)}}} 
 d\rho\,dr
\end{equation}
defines a bounded functional ${\Phi_{g_1,g_2}}: \cB_{1,rad} \times \cB_{1,rad} \to [0,\infty)$. 
\end{lemma}

\begin{proof}
The result follows directly from \cite[Lemma 2.10]{CiWe2}, applied with $\beta_1 = \beta_2 = -\frac{N}{2(N-1)}$.
\end{proof}

\begin{lemma}
	\label{C-B-1-restimate}
	Suppose that $g$ satisfies $(g_0)$ and $C_g \in [0, \infty)$. Then the functional $\Phi_-$ is uniformly bounded on $\cB_1$, i.e., we have $m_1^-(F)= \sup \limits_{\cB_1}\Phi^-< \infty$.   
\end{lemma}

\begin{proof}
	Let $u \in \cB_1$. Then we have 
	$$
	\Phi_-(u)= b_-(F(u),F(u)) = \int_{B_1} \int_{B_1} \ln^- \!\frac{1}{|x-y|}\, F(u(x))F(u(y))\,dx dy \le (\ln 2) \|F(u)\|_{L^1(B_1)}^2
	$$
	where $\|F(u)\|_{L^1(B_1)} \le c_1 \int_{B_1}e^{\alpha_N |u|^{\frac{N}{N-1}}}\,dx \le c_2$
	with constants $c_1,c_2>0$ independent of $u$ by assumption and  the Trudinger-Moser inequality.
\end{proof}

Now Theorem~\ref{sec:introduction-main-thm-last}(i) is a direct consequence of the following Proposition.

\begin{proposition}
	\label{C-finiteness}
	Suppose that $g$ satisfies $(g_0)$ and $C_g \in [0, \infty)$. Then we have 
	$$
	m_1(F) \le m_1^+(F)< \infty,\qquad \text{where}\qquad  m_1^+(F)= \sup_{\cB_1}\Phi^+. 
	$$
\end{proposition}

\begin{proof}
We note that $F$ has at most $(-\frac{N}{2(N-1)})$-critical growth if $C_g< \infty$, and therefore the assertion is already proved in \cite[Prop. 4.2]{CC} under the additional assumption that $F$ is increasing. So it remains to prove $m_1^+(F)< \infty$ without this extra assumption. Since $g$ satisfies $(g_0)$ and $C_g < \infty$, we can choose $\kappa >0$ large enough such that 
  $$
  F(t) \le F_\kappa(t):= \kappa \frac{1+ e^{1-t}t}{(1+t)^{\frac{N}{2(N-1)}}}e^{\alpha_N |t|^{\frac{N}{N-1}}}
  $$
  Moreover, the function $F_\kappa$ is increasing
  and also has at most $(-\frac{N}{2(N-1)})$-critical growth, so we have
  $m_1^+(F_\kappa) <\infty$ by \cite[Prop. 4.2]{CC}.
Moreover, 
\begin{align*}
  \Phi^+(u)=  \int_{B_1} \int_{B_1} \ln^+ \frac{1}{|x-y|}\, F(u(x))F(u(y))\,dx dy &\le  \int_{B_1} \int_{B_1} \ln^+ \frac{1}{|x-y|}\, F_\kappa(u(x))F_\kappa(u(y))\,dx dy\\
       &\le m_1^+(F_\kappa)\qquad \text{for $u \in \cB_1$,}
\end{align*}
which shows the required finiteness of $m_1^+(F)$.
\end{proof}

\section{The second part of Theorem \ref{sec:introduction-main-thm-last}: the unbounded case.}
\label{sec:unbounded-case}

In this section we shall complete the proof of Theorem~\ref{sec:introduction-main-thm-last}(ii), which we restate in the following Proposition.

\begin{proposition}
\label{C-infiniteness}
Suppose that $g$ satisfies $(g_0)$ and 
 $\lim \limits_{s \to + \infty} {g(s)}= \infty$.
 Then there exists a sequence of functions $u_n \in \cB_1 \cap L^\infty(B_1)$ with $\Phi(u_n) \to \infty$ as $n \to \infty$.
\end{proposition}

\begin{proof}
Set $v_n = F(u_n)$.
For $n \in \N$, 
${n \geq 2}$, we now define 
 $u_n= m_n \in W^{1,N}_0(B_1) \cap L^\infty(B_1)$ as in \cite[p. 309]{Doo-Marcos}, namely
\[
m_n:=
\begin{cases}
	\frac{1}{{\omega^{1/N}_{N-1}}} {(\ln n)^{(N-1)/N}},  \quad  \  0 \leq |x| \leq \frac{1}{n}, \\
\frac{1}{{\omega^{1/N}_{N-1}}} \frac{\ln(\frac{1}{\abs{x}})}{(\ln n)^{1/N}}, \quad  \,\,\,\, \,\,\,\,\, \,\,\, \ \frac{1}{n} \leq |x| \leq 1.
\end{cases}
\]
 
As noted in \cite[p. 310]{Doo-Marcos}, we then have $|\nabla u_n|_N \le 1$ for $n$ large and thus $u_n \in \cB_1^*$. Moreover, $v_n :=F(u_n) \in L^\infty(B_1)$ and therefore 
$$
\Phi_\pm(u_n) = b_\pm(v_n,v_n) < \infty \qquad \text{for $n \in \N$.}
$$
We have, for $n$ large,  
$$
v_n = g((\frac{(\ln n)^{N-1}}{\omega_{N-1}})^{\frac{1}{N}}) e^{\alpha_N \Big(\frac{(\ln n)}{\omega^{1/(N-1)}_{N-1}}\Big)}\:\ge\: 
 c_{1} \frac{g((\frac{(\ln n)^{N-1}}{\omega_{N-1}})^{\frac{1}{N}})}{(\ln n)^{\frac{1}{2}}}n^{N}
\qquad \text{on $B_{\frac{1}{n}}(0)$}
$$
with some constant  $c_1>0$. We derive that
\begin{align*}
\frac{b_0(v_n,v_n)}{\omega_{N-1}^2}& \geq  2\int_{0}^{\frac{1}{n}} r^{N-1} v_n(r) \ln \frac{1}{r} \int_{0}^r \rho^{N-1}  v_n(\rho)d\rho dr \\ & \ge 2 c^2_{1} \frac{g^2((\frac{(\ln n)^{N-1}}{\omega_{N-1}})^{\frac{1}{N}})}{\ln n}n^{2N}
 \int_{0}^{\frac{1}{n}} r^{N-1} \ln \frac{1}{r} \int_{0}^{r} \rho^{N-1} d\rho dr\\
&= \frac{2}{N} c^2_{1} \frac{g^2((\frac{(\ln n)^{N-1}}{\omega_{N-1}})^{\frac{1}{N}})}{\ln n}n^{2N}
  \int_{0}^{\frac{1}{n}} r^{2N-1}\ln \frac{1}{r} dr \\
&= \frac{2}{N} c^2_{1} \frac{g^2((\frac{(\ln n)^{N-1}}{\omega_{N-1}})^{\frac{1}{N}})}{\ln n}n^{2N} \frac{1}{2Nn^{2N}}\Big(\ln n +\frac{1}{2N}\Big)\\ &
\ge \frac{1}{{N^2}} c_1^2 g^2((\frac{(\ln n)^{N-1}}{\omega_{N-1}})^{\frac{1}{N}})
\end{align*}

so that $b_0(v_n,v_n) \to \infty$ as $n \to \infty$.  This shows that $\Phi(u_n) = b_0(v_n,v_n)  \to  \infty$ as $n \to \infty$, 
as required. 
\end{proof}

\section{the proof of Theorem \ref{sec:introduction-main-thm-last}: the existence of maximizers in the subcritical case}
\label{sec:cont-prop-phi}
In this section we provide an abstract strong continuity result for the functional $\Phi$, which is partly based on \cite[Theorem 1.6]{Lions}. Moreover, we will complete the proof of Theorem~\ref{sec:introduction-main-thm-last}. 

\begin{proposition}
\label{abstract-strong-continuity}
Suppose that $g$ satisfies $(g_0)$ and $(g_1)$, and let $(u_n)_n$ be a sequence in $\cB_1^*$ with $u_n \weak u$ in $W^{1,N}_0(B_1)$. Suppose moreover that \underline{one} of the following conditions is satisfied:
\begin{itemize}
\item[(i)] $u \not =0$.
\item[(ii)] $C_g=0$, i.e., $\lim \limits_{t \to \infty}g(t)=0$.
\end{itemize}
 Then we have 
$$
\lim_{n \to \infty}\Phi(u_n)= \Phi(u).
$$
\end{proposition}

\begin{proof}
Since $u_n \in \cB_1^*$ for all $n$, it is easy to deduce from the weak convergence $u_n \weak u$ in $W^{1,N}_0(B_1)$ that $u \in \cB_1^*$.

We now assume (i) first, so we assume that $u \not = 0$. Then \cite[Theorem 1.6]{Lions} implies that 
  \begin{equation}
    \label{eq:ii-l-supercritical-boundedness}
\int_{B_1}  e^{(\alpha_N +t )u_n^{\frac{N}{N-1}} } \ dx \qquad \text{is bounded for some $t>0$,}
\end{equation}
 and thus 
  \begin{equation}
    \label{eq:ii-l-1-convergence}
e^{\alpha_N u_n^{\frac{N}{N-1}}} \to e^{\alpha_N u^{\frac{N}{N-1}} }  \qquad \text{in $L^1(B_1)$.}
  \end{equation}
Set $v_n:= 1_{B_1} F(u_n)$ for $n \in \N$ and $v:= 1_{B_1}F(u)$. By \eqref{eq:ii-l-supercritical-boundedness},  $v_n$ is bounded in  $L^{s_0}(\R^N)$ with $s_0=1 + \frac{t}{\alpha_N}>1$. Moreover, since 
$$
v_n \to v\qquad \text{in $L^1(B_1)$,}
$$
interpolation yields that 
$$
v_n \to v  \qquad \text{in $L^s(\R^N)\qquad$ for $1 \le s < s_0$.}
$$
Moreover,  
\begin{align*}
&\frac{\Phi(u_n)-\Phi(u)}{2\omega_{N-1}^2}= \int_{0}^1 r^{N-1} v_n(r) \ln \frac{1}{r} \int_{0}^r \rho^{N-1}  v_n(\rho) d\rho dr - \int_{0}^1 r^{N-1} v(r) \ln \frac{1}{r} \int_{0}^r \rho^{N-1}  v(\rho) d\rho dr\\
&= \int_{0}^1 r^{N-1} v_n(r) \ln \frac{1}{r} \int_{0}^r \rho^{N-1}  [v_n(\rho)-v(\rho)] d\rho dr + \int_{0}^1 r^{N-1} [v_n(r)-v(r)] \ln \frac{1}{r} \int_{0}^r \rho^{N-1}  v(\rho) d\rho dr
\end{align*}
where, for fixed $s \in (1,s_0)$,  
\begin{align*}
&\Bigl|\int_{0}^1 r^{N-1} v_n(r) \ln \frac{1}{r} \int_{0}^r \rho^{N-1}  [v_n(\rho)-v(\rho)] d\rho dr\Bigr|\le \frac{|v_n-v|_{s}}{\omega_{N-1}} \int_{0}^1 r^{N-1} |B_r|^{\frac{1}{s'}}  v_n(r) \ln \frac{1}{r}  dr  \\
&=\frac{|v_n-v|_{s}}{\omega_{N-1}} \int_{0}^1 r^{N-1} \bigg|\frac{r^N \pi^{\frac{N}{2}}}{\Gamma(\frac{N}{2}+1 )}\bigg|^{\frac{1}{s'}}  v_n(r) \ln \frac{1}{r}  dr\\
&\le \frac{|v_n-v|_{s} \pi^{\frac{N}{2s'}}}{\Gamma(\frac{N}{2}+1)\omega_{N-1}} 
\int_{0}^1 r^{N-1+\frac{N}{s'}} v_n(r) \ln \frac{1}{r} dr \le C |v_n-v|_{s} |v_n|_{1}  \to 0 \qquad \text{as $n \to \infty$ } 
\end{align*}
with $C:= \frac{\pi^{{\frac{N}{2s'}}}}{\Gamma(N/2+1)\omega_{N-1}^2}\sup \limits_{r \in (0,1]} r^{\frac{N}{s'}} \ln \frac{1}{r}$  and also 
\begin{align*}
&\Bigl|\int_{0}^1 r^{N-1} [v_n(r)-v(r)] \ln \frac{1}{r} \int_{0}^r \rho^{N-1}  v(\rho) d\rho dr\Bigr| \le \frac{|v|_{s}}{\omega_{N-1}} \int_{0}^1 r^{N-1} |B_r|^{\frac{1}{s'}}  [v_n(r)-v(r)] \ln \frac{1}{r}  dr  \\
&\le \frac{|v|_{s}\pi^{\frac{N}{2s'}}}{\Gamma(\frac{N}{2}+1)\omega_{N-1}} 
\int_{0}^1 r^{N-1+\frac{N}{2s'}} [v_n(r)-v(r)] \ln \frac{1}{r} dr \le C |v|_{s} |v_n-v|_{1} \to 0 \qquad \text{as $n \to \infty$.} 
\end{align*}
We thus conclude that 
\begin{equation*}
m_1(F) = \lim_{n \to \infty} \Phi(u_n) = \Phi(u).
\end{equation*}
Next we assume (ii), and by (i) we may assume that $u=0$, so $u_n \weak 0$ in $W^{1,N}_0(B_1)$. Since $W^{1,N}_0(B_1)$ is compactly embedded into $L^p(B_1)$ for $2<p<\infty$, we have 
\begin{equation}
  \label{eq:l-p-strong-zero}
u_n \to 0 \qquad \text{in $L^p(B_1)$ for $2 < p< \infty$.}
\end{equation}
Since $u_n \in \cB_1^*$ for every $n \in \N$, (\ref{eq:l-p-strong-zero}) implies that 
\begin{equation}
  \label{eq:l-p-strong-zero-locally-uniform}
u_n \to 0 \qquad \text{uniformly in $[\delta,1]$ for every $\delta  \in (0,1)$.}
\end{equation}
We now write $F= \kappa_0 + \tilde F$ with $\kappa_0=F(0)$, where the function $\tilde F = F- \kappa_0$ is also even, nonnegative
and increasing on $[0,\infty)$. Moreover, it satisfies $\tilde F(0)=0$ and 
\begin{equation}
  \label{eq:beta-critical-estimate}
\tilde F(t)\le  c_1 e^{\alpha_N |t|^{\frac{N}{N-1}}} \qquad \text{for $t \in \R$ with a constant $c_1>0$.}
\end{equation}
With
$$
v_n:= \tilde F(u_n) \qquad \text{for $n \in \N$,}
$$
we then have
\begin{equation}
  \label{eq:Phi-u-n-u-relation}
\Phi(u_n)= b_0(v_n) + 2 b_0 \bigl(1_{B_1}\kappa_0,v_n\bigr) + b_0(\kappa_0 1_{B_1})
= b_0(v_n) + 2 b_0 \bigl(1_{B_1} \kappa_0,v_n\bigr) + \Phi(0).
\end{equation}
By (\ref{newton-2-arguments}) we have 
$$
b_0 \bigl(1_{B_1} \kappa_0,v_n \bigr) = \omega_{N-1}^2 \kappa_0 \int_{0}^1 r^{N-1} v_n(r) h(r)dr \quad \text{with}\quad 
h(r)= \ln \frac{1}{r}  \int_{0}^r \rho^{N-1} d\rho + \int_{r}^1 \rho^{N-1} (\ln \frac{1}{\rho}) d\rho. 
$$
Moreover, for any $\delta \in (0,1)$ we have, by T-M inequality and (\ref{eq:beta-critical-estimate}), 
\begin{equation}
  \label{eq:g-L-infty-delta-est}
\Bigl| \int_{0}^\delta r^{N-1} v_n h(r)dr \Bigr| \le \frac{c_1}{\omega_{N-1}} \|h\|_{L^\infty(0,\delta)}\int_{B_1}e^{\alpha_N u_n^{\frac{N}{N-1}}} dx \le \frac{c_1 c(B_1)}{\omega_{N-1}} \|h\|_{L^\infty(0,\delta)}.
\end{equation}
By (\ref{eq:l-p-strong-zero-locally-uniform}) and since $\tilde F(0)=0$, we also have
\begin{equation}
  \label{eq:delta-uniform-convergence}
v_n \to 0 \qquad \text{uniformly in $[\delta, 1]$ for every $\delta \in (0,1)$.}
\end{equation}
Combining (\ref{eq:g-L-infty-delta-est}), (\ref{eq:delta-uniform-convergence}) and the fact that $h(r) \to 0$ as $r \to 0$, we see that  
$$
b_0 \bigl(1_{B_1} \kappa_0,v_n \bigr) \to 0 \qquad \text{as $n \to \infty$.}
$$
To prove that $\Phi(u_n) \to \Phi(0)$ as $n \to \infty$, it thus remains, by (\ref{eq:Phi-u-n-u-relation}), to show that 
\begin{equation}
  \label{limit}
b_0(v_n) \to 0 \qquad \text{as $n \to \infty$.}  
\end{equation}
To see (\ref{limit}), we note that for every $\delta \in (0,1)$ we have
$$
\frac{b_0(v_n)}{2 \omega_{N-1}^2} = \int_{0}^1 r^{N-1} v_n(r) \ln \frac{1}{r} \int_{0}^r \rho^{N-1}  v_n(\rho)d\rho dr = M_n^\delta + N_n^\delta,
$$
where, by T-M inequality, \eqref{eq:beta-critical-estimate} and (\ref{eq:delta-uniform-convergence})
\begin{align}
M_n^\delta &:= \int_{\delta}^1 r^{N-1}v_n(r) \ln \frac{1}{r} \int_{0}^r \rho^{N-1}  v_n(\rho)d\rho dr \\ & \nonumber \le c_1 \int_{\delta}^1 r^{N-1}v_n(r) \ln \frac{1}{r} \int_{0}^1 \rho^{N-1} e^{\alpha_N u_n^{\frac{N}{N-1}}(\rho)} d\rho dr \nonumber\\
&\le \frac{c_1 c(B_1)}{\omega_{N-1}} \int_{\delta}^1 r^{N-1}v_n(r) \ln \frac{1}{r} dr \to 0 \qquad \text{as $n \to \infty$.} \label{b-n-v-n-zero-est-1}
\end{align}
To estimate 
$$
N_n^\delta:= \int_{0}^\delta r^{N-1} v_n(r) \ln \frac{1}{r} \int_{0}^r \rho^{N-1}  v_n(\rho)d\rho dr 
$$
we fix $\eps \in (0,\frac{1}{\alpha_N})$ and define, for any $n \in \N$,
$$
A_n^+
= \{r \in (0,1]\::\: u_n(r) \ge  [\epsilon(-\ln r)]^{\frac{N-1}{N}} \},\qquad
A_n^-
:= \{r \in (0,1]\::\: u_n(r) < [\eps(-\ln r)]^{\frac{N-1}{N}} \}.
$$
Since $\tilde F$ is an increasing function and $g$ is bounded on $[0,\infty)$ as a consequence of the assumptions $C_g=0$, we then have 
\begin{align}
	\label{eq:A--estimatess}
	v_n(r) &\le \tilde F([\eps(-\ln r)]^{\frac{N-1}{N}})\\ & \nonumber \le F([\eps(-\ln r)]^{\frac{N-1}{N}}) = \frac{g([\eps(-\ln r)]^{\frac{N-1}{N}})e^{-\alpha_N \eps \ln r}}{(1+ [\eps(-\ln r)]^{\frac{N-1}{N}}})^{\frac{N}{2(N-1)}} \le M r^{-\alpha_N \eps}\qquad \text{for $r \in A_n^-$,}
\end{align}
with some constant $M>0$ and
\begin{equation}
	\label{eq:A-+estimatess}
	v_n(r) \le  \frac{g(u_n(r)) e^{\alpha_N u_n^{\frac{N}{N-1}}(r)}
	}{(1+ [\eps(-\ln r)]^{\frac{N-1}{N}}          )^{\frac{N}{2(N-1)}}} \le  \frac{C_\eps(r) e^{\alpha_N u_n^{\frac{N}{N-1}}(r)}
	}{(1+ [\eps(-\ln r)]^{\frac{N-1}{N}})^{\frac{N}{2(N-1)}}}  \qquad \text{for $r \in A_n^+$}
\end{equation}
with the increasing function 
$$
r \mapsto C_\eps(r):= \sup \{g(t)\::\: t \ge [\epsilon(-\ln r)]^{\frac{N-1}{N}}\}.
$$
In particular, we thus have
\begin{equation}
	\label{eq:A-+estimatess-special}
	v_n(r) \le C_\eps(1)e^{\alpha_N u_n^{\frac{N}{N-1}}(r)}
 \qquad \text{for $r \in A_n^+$.}
\end{equation}
We now write 
\begin{align*}
N_n^\delta= \int_{A_n^- \cap (0,\delta)} r^{N-1} v_n(r) \ln \frac{1}{r}
\int_{0}^r \rho^{N-1}  v_n(\rho)d\rho dr + \int_{A_n^+ \cap (0,\delta)} r^{N-1} v_n(r) \ln \frac{1}{r}\int_{0}^r \rho^{N-1}  v_n(\rho)d\rho dr,
\end{align*}
where, by (\ref{eq:A--estimatess}) and since $\eps \in (0,\frac{1}{\alpha_N})$,
\begin{align}
&\int_{A_n^- \cap (0,\delta)} r^{N-1} v_n(r) \ln \frac{1}{r}
\int_{0}^r \rho^{N-1}  v_n(\rho)d\rho dr  \le M^2 \int_{0}^\delta r^{N-1-\alpha_N \eps} \ln \frac{1}{r} \int_{0}^1 \rho^{N-1-\alpha_N \eps} d\rho dr \nonumber\\
&\le  M^2 \int_{0}^\delta r^{N-1-\alpha_N \eps} \ln \frac{1}{r}dr = M^2 \Bigl(\frac{\delta^{N-\alpha_N \eps}}{N-\alpha_N \eps}-\frac{\delta^{N-\alpha_N \eps}\ln \delta }{N-\alpha_N \eps}\Bigr) 
\label{N-delta-1-est}
\end{align}
for all $n \in \N$. Moreover, we have 
\begin{align*}
\int_{A_n^+ \cap (0,\delta)} r^{N-1} v_n(r) \ln \frac{1}{r}\int_{0}^r \rho^{N-1}  v_n(\rho)d\rho dr& = \int_{A_n^+ \cap (0,\delta)} r^{N-1} v_n(r) \ln \frac{1}{r}\int_{A_n^+ \cap (0,r)} \rho^{N-1}  v_n(\rho)d\rho dr \\ & + \int_{A_n^+ \cap (0,\delta)} r^{N-1} v_n(r) \ln \frac{1}{r}\int_{A_n^- \cap (0,r)} \rho^{N-1}  v_n(\rho)d\rho dr,
\end{align*}
where, by (\ref{eq:A-+estimatess}) 
\begin{align}
&\int_{A_n^+ \cap (0,\delta]} r^{N-1}v_n(r) \ln \frac{1}{r} \int_{A_n^+ \cap [0,r]} \rho^{N-1}  v_n(\rho)d\rho dr \nonumber\\
&\le \int_{A_n^+ \cap (0,\delta]}
r^{N-1} \frac{C_\eps(r) e^{\alpha_N u_n^{\frac{N}{N-1}}(r)}}{(1+ [\eps(-\ln r)]^{\frac{N-1}{N}})^{\frac{N}{2(N-1)}}}
\ln \frac{1}{r} 
\int_{A_n^+ \cap [0,r]} \rho 
 \frac{C_\eps(\rho) e^{\alpha_N u_n^{\frac{N}{N-1}}(\rho)}
}{(1+ [\eps(-\ln \rho)]^{\frac{N-1}{N}})^{\frac{N}{2(N-1)}}} d\rho dr \nonumber\\
&\le C_\eps(\delta)^2  \int_{A_n^+ \cap (0,\delta]}
\frac{(-\ln r) r e^{\alpha_N u_n^{\frac{N}{N-1}}(r)}}{(1+ [\eps(-\ln r)]^{\frac{N-1}{N}})^{\frac{N}{N-1}}} \int_0^r  
\rho^{N-1} e^{\alpha_N u_n^\frac{N}{N-1}(\rho)}d\rho dr \nonumber\\
&\le \frac{C_\eps(\delta)^2}{\eps} 
\Bigl(\int_{0}^1  
r e^{\alpha_N u_n^\frac{N}{N-1}(r)} dr\Bigr)^2 \le \frac{\bigl(c(B_1)\bigr)^2}{\omega_{N-1}^2 \eps}C_\eps(\delta)^2 \label{N-delta-2-est}
\end{align}
by the classical T-M inequality. Furthermore, by (\ref{eq:A--estimatess}) and (\ref{eq:A-+estimatess-special}),
\begin{align}
& \nonumber \int_{A_n^+ \cap [0,\delta]}
r^{N-1} v_n(r) \ln \frac{1}{r} \int_{A_n^- \cap [0,r]} \rho^{N-1}  v_n(\rho)d\rho dr \\ &\leq 
                C_\eps(1)M \int_{A_n^+ \cap [0,\delta]} {r^{N-1} e^{\alpha_N u_n^{\frac{N}{N-1}}(r)}}
\ln \frac{1}{r}  \int_{0}^r 
\rho^{N-1- \alpha_N \eps} d\rho dr \nonumber\\
  & \nonumber \le \frac{C_\eps(1)M}{N- \alpha_N \eps} \int_{0}^\delta r^{2N-1- \alpha_N \eps} e^{\alpha_N u_n^{\frac{N}{N-1}}(r)} \ln \frac{1}{r} \,dr\\ & \leq
C_\eps(1)M  \sup_{s \in [0,\delta]}\Bigl(s^{N- \alpha_N \eps}\ln \frac{1}{s}\Bigr)
    \int_{0}^1 r^{N-1} e^{\alpha_N u_n^2(r)}\,dr \nonumber\\
  & \le \frac{C_\eps(1)M c(B_1) }{2\pi}\sup_{s \in [0,\delta]}\Bigl(s^{N- \alpha_N \eps}\ln \frac{1}{s}\Bigr)  
\label{N-delta-3-est}
\end{align}
again by the T-M inequality. We observe now that $C_\delta \to 0$ as $\delta \to 0$ by assumption (ii). So, as $\eps \in (0,\frac{1}{\alpha_N})$  the RHS of (\ref{N-delta-1-est}), (\ref{N-delta-2-est}) and (\ref{N-delta-3-est}) tend to zero as $\delta \to 0^+$. Hence we infer that
$$
\lim_{\delta \to 0}\sup_{n \in \N}N_n^\delta =0.
$$
Combining this with (\ref{b-n-v-n-zero-est-1}), we infer (\ref{limit}), as claimed.

\end{proof}

The following Proposition completes the proof of Theorem~\ref{sec:introduction-main-thm-last}.

\begin{proposition}
  \label{C-existence-of-maximizer}
  Suppose that $g$ satisfies $(g_0)$ and $(g_1)$ with $C_g= 0$. Then the value 
$m_1(F)< \infty$ is attained by a function $u \in \cB_1^*$. 
\end{proposition}

\begin{proof}
  Let $(u_n)_n$ be a maximizing sequence in $\cB_1$ for $m_F$. By Polya-Szego inequality, we may assume that $u_n \in \cB_1^*$ for $n \in \N$. Since $\cB_1$ is bounded in $W^{1,N}_0(B_1)$, we may also assume that $u_n \weak u \in W^{1,N}_0(B_1)$ with $u \in \cB_1^*$. By Proposition~\ref{abstract-strong-continuity}, we then have
  $$
  m_1(F)= \lim_{n \to \infty}\Phi(u_n)= \Phi(u),
  $$
  so $m_1(F)$ is attained at $u \in \cB_1^*$.
\end{proof}

\section{An upper bound for Schwarz symmetric concentrating sequences}
\label{sec:an-upper-bound}

We devote this section to the proof of the first assertion of Theorem~\ref{sec:introduction-CC-theorem}, concerning the asymptotic upper bound for Schwarz symmetric concentrating sequences (\ref{limsup-SCS}). Exploiting this bound together with the continuity criterion from Proposition~\ref{abstract-strong-continuity}, we will be able to finalize the proof of Theorem~\ref{sec:introduction-main-thm-critical-sufficient-cond}. A key tool in this analysis is a particular change of variables, adapted from \cite{carlesonchang}, defined as follows.

\begin{lemma}
  \label{transformation}
  Let $u \in W_0^{1,N}(B_1)$ be a radial nonnegative and radially decreasing function, and let
    $w: [0,\infty) \to \R$ be defined by $w(t) = N^\frac{N-1}{N} \omega^\frac{1}{N}_{N-1} u(e^{-t/N})$ (where we identify $u$ with its profile function in the radial variable). Then $w \in W_{loc}^{1,N}(\R^+)$ is an increasing function with $w(0)=0$. Moreover, we have
  \begin{equation}
    \label{eq:transformation-Dirichlet-integral}
    \int_0^\infty |w^\prime (t)|^N dt = \omega_{N-1} \int_0^1 |u^\prime (r)|^N r^{N-1} dr
    = \int_{B_1} | \nabla u |^N dx,
  \end{equation}
and
  \begin{equation}\footnotesize 
    \label{eq:functional}
\Phi (u) = \frac{2 \omega_{N-1}^2}{N^3} \int_0^\infty \frac{y e^{w^{N/(N-1)}(y) - y} g( N^\frac{1-N}{N} \omega_{N-1}^{-\frac{1}{N}} w(y) )  }{ ( 1+N^\frac{1-N}{N} \omega_{N-1}^{-\frac{1}{N}} w(y) )^\frac{N}{2(N-1)} } \int_{y}^\infty \frac{ e^{ w^{N/(N-1)} (x) - x} g( N^\frac{1-N}{N} \omega_{N-1}^{-\frac{1}{N}} w(x) ) }{( 1+N^\frac{1-N}{N} \omega_{N-1}^{-\frac{1}{N}} w(x) )^\frac{N}{2(N-1)}} dx dy.
\end{equation}
\end{lemma}

\begin{proof}
By definition, $w$ is an increasing function satisfying $w(0)=0$, and (\ref{eq:transformation-Dirichlet-integral}) follows by a straightforward computation. Moreover, since 
  $$
\Phi(u) = 2 \omega_{N-1}^2 \int_0^1 r^{N-1} F(u(r)) \ln \frac{1}{r} \int_0^r \rho^{N-1} F(u(\rho)) d\rho dr
$$
by (\ref{newton-1-argument}), we have, by the change of variables $r= e^{-t/N}$, 
$$
\begin{aligned} \footnotesize
\Phi (u) &= \frac{2 \omega_{N-1}^2}{N^3} \int_0^\infty t e^{-t} F (N^\frac{1-N}{N} \omega_{N-1}^{-\frac{1}{N}} w(t)) \int_0^\infty e^{-(s+t)} F(N^\frac{1-N}{N} \omega_{N-1}^{-\frac{1}{N}} w(s+t)) ds dt \\
&\footnotesize = \frac{2 \omega_{N-1}^2}{N^3} \int_0^\infty \frac{t e^{w^{N/(N-1)}(t) - t} g( N^\frac{1-N}{N} \omega_{N-1}^{-\frac{1}{N}} w(t) ) }{ ( 1+N^\frac{1-N}{N} \omega_{N-1}^{-\frac{1}{N}} w(t) )^\frac{N}{2(N-1)} }  \\ & \times \int_0^\infty \frac{ e^{w^{N/(N-1)} (s+t) - (s+t)} g( ( N^\frac{1-N}{N} \omega_{N-1}^{-\frac{1}{N}} w(s+t) ) }{( 1+N^\frac{1-N}{N} \omega_{N-1}^{-\frac{1}{N}} w(s+t) )^\frac{N}{2(N-1)}} ds dt\\
&= \frac{2 \omega_{N-1}^2}{N^3} \int_0^\infty \frac{y e^{w^{N/(N-1)}(y) - y} g( N^\frac{1-N}{N} \omega_{N-1}^{-\frac{1}{N}} w(y) )  }{ ( 1+N^\frac{1-N}{N} \omega_{N-1}^{-\frac{1}{N}} w(y) )^\frac{N}{2(N-1)} } \\ & \times \int_{y}^\infty \frac{ e^{ w^{N/(N-1)} (x) - x} g( N^\frac{1-N}{N} \omega_{N-1}^{-\frac{1}{N}} w(x) ) }{( 1+N^\frac{1-N}{N} \omega_{N-1}^{-\frac{1}{N}} w(x) )^\frac{N}{2(N-1)}} dx dy.
\end{aligned}
$$
\end{proof}

We may now complete the
\begin{proof}[Proof of Theorem~\ref{sec:introduction-CC-theorem}(i)] 
  Let $(u_n)_n \subset \mathcal{B}_1 $ be any SCS-sequence, and let $w_n:[0,\infty) \to \R$ be defined by $ w_n(t) = N^\frac{N-1}{N} \omega^\frac{1}{N}_{N-1} u_n(e^{-t/N}) $ for $n \in \N$.   By Lemma~\ref{transformation}, we then have $w_n \in W^{1,N}_{loc}(\R_+)$ with $w_n(0) = 0 $,
  \begin{equation}
    \label{eq:local-zero-deriv-conv-0}
\int_0^\infty (w^\prime_n)^N dt \le 1  \qquad \text{for all $n \in \N$}
\end{equation}
and
  \begin{equation}
    \label{eq:local-zero-deriv-conv}
\int_0^A w^\prime_n dt \to 0 \qquad \text{as $n \to \infty$ for every $A>0$.}
\end{equation}
Moreover, $w_n$ is nondecreasing for every $n$, and by (\ref{eq:functional}) it suffices to show that
\begin{equation}\label{upper_bound}
  \begin{split}
  \limsup_{n \to \infty} \int_0^{+\infty} \int_y^{+\infty} \frac{y e^{w_n^{N/(N-1)}(y) - y} g( N^\frac{1-N}{N} \omega_{N-1}^{-\frac{1}{N}} w_n(y) )  }{ ( 1+N^\frac{1-N}{N} \omega_{N-1}^{-\frac{1}{N}} w_n(y) )^\frac{N}{2(N-1)} } \frac{ e^{ w_n^{N/(N-1)} (x) - x} g( N^\frac{1-N}{N} \omega_{N-1}^{-\frac{1}{N}} w_n(x) ) }{( 1+N^\frac{1-N}{N} \omega_{N-1}^{-\frac{1}{N}} w_n(x) )^\frac{N}{2(N-1)}} dx dy\\
  \leq \frac{{g^2(0)}}{4} + \frac{N}{2}  \omega_{N-1}^\frac{1}{N-1} C_g^2 e^{2(1+\frac{1}{2}+\cdots+ \frac{1}{N-1})  } .
  \end{split}
\end{equation}

To see this, we first note that \eqref{eq:local-zero-deriv-conv-0} yields  
  \begin{equation}
    \label{eq:local-uniform-bound}
 w_n^\frac{N}{N-1}(t) = \Bigl( \int_0^t w_n'(s)\,ds \Bigr)^\frac{N}{N-1} \le t \left( \int_0^\infty (w^\prime_n)^N \right)^{1/(N-1)} \le t \qquad \text{for every $t \ge 0$,}     
\end{equation}
while \eqref{eq:local-zero-deriv-conv} gives
  \begin{equation}
    \label{eq:local-uniform-conv}
    w_n \to 0 \qquad \text{locally uniformly on $[0,\infty)$.}     
\end{equation}
We now define, for $n \in \N$,
$$
a_n:= \inf\{ t \in [3,\infty)\::\: w_n^\frac{N}{N-1}(t) \ge t - 3 \log t \} \quad \text{in $[3,\infty]$.}
$$
So if $a_n = \infty$, then $w_n^\frac{N}{N-1}(t) \le t - 3 \log t$ for all $t \ge 3$, while $a_n$ is the first point $a_n \in [1, +\infty)$ with $ w_n^\frac{N}{N-1} (a_n) = a_n - 3 \log a_n $ if $a_n< \infty$. We also note that
\begin{equation}
  \label{eq:a-n-infty}
a_n \to \infty \qquad \text{as $n \to \infty$}  
\end{equation}
by (\ref{eq:local-uniform-conv}). Moreover, we have the estimate
\begin{equation}
  \label{eq:cc-estimate}
\limsup_{n \to \infty} \int_{a_n}^\infty e^{ w_n^{N/(N-1)} (x) - x}\,dx \le e^{1+\frac{1}{2}+\cdots+\frac{1}{N-1}}.
\end{equation}
Indeed, this is true by \cite[P. 140]{figuereido-do-o-ruf} if $a_n$ is replaced by $\tilde a_n = \inf\{ t \in [1,\infty)\::\: w_n^\frac{N}{N-1}(t) \ge t - 2 \log t \}$.
Note that $a_n \le \tilde a_n$, and, moreover
$$
\limsup_{n \to \infty} \int_{a_n}^{\tilde a_n}  e^{ w_n^{N/(N-1)} (x) - x}\,dx \le
\limsup_{n \to \infty} \int_{a_n}^{\tilde a_n} \frac{1}{x^2} \,dx = \frac{1}{a_n}-\frac{1}{\tilde a_n} \to 0\quad \text{as $n \to \infty$,}
$$
where we have used the convention $\frac{1}{\infty}=0$. Hence (\ref{eq:cc-estimate}) holds.

Next we show that
\begin{equation} \label{limit-g-0}
\limsup_{n \to \infty} \int_0^{a_n} \frac{y e^{w_n^{N/(N-1)}(y) - y} g( N^\frac{1-N}{N} \omega_{N-1}^{-\frac{1}{N}} w_n(y) )  }{ ( 1+N^\frac{1-N}{N} \omega_{N-1}^{-\frac{1}{N}} w_n(y) )^\frac{N}{2(N-1)} } dy \le g(0).
\end{equation}
On the one hand, we have, for fixed $ A > 0 $,
$$
\begin{aligned}
&\int_0^{a_n}  \frac{y e^{w_n^{N/(N-1)}(y) - y} g( N^\frac{1-N}{N} \omega_{N-1}^{-\frac{1}{N}} w_n(y) )  }{ ( 1+N^\frac{1-N}{N} \omega_{N-1}^{-\frac{1}{N}} w_n(y) )^\frac{N}{2(N-1)} } dy \\&= \int_0^{A}  \frac{y e^{w_n^{N/(N-1)}(y) - y} g( N^\frac{1-N}{N} \omega_{N-1}^{-\frac{1}{N}} w_n(y) )  }{ ( 1+N^\frac{1-N}{N} \omega_{N-1}^{-\frac{1}{N}} w_n(y) )^\frac{N}{2(N-1)} } dy+ \int_A^{a_n}  \frac{y e^{w_n^{N/(N-1)}(y) - y} g( N^\frac{1-N}{N} \omega_{N-1}^{-\frac{1}{N}} w_n(y) )  }{ ( 1+N^\frac{1-N}{N} \omega_{N-1}^{-\frac{1}{N}} w_n(y) )^\frac{N}{2(N-1)} } dy
\end{aligned}
$$
where, since $w_n \to  0$ uniformly on $[0,A]$,
\begin{align*}
 &\int_0^{A}  \frac{y e^{w_n^{N/(N-1)}(y) - y} g( N^\frac{1-N}{N} \omega_{N-1}^{-\frac{1}{N}} w_n(y) )  }{ ( 1+N^\frac{1-N}{N} \omega_{N-1}^{-\frac{1}{N}} w_n(y) )^\frac{N}{2(N-1)} } dy \\ &\to g(0) \int_0^A y e^{-y}\,dy = g(0) ( 1 - (1+A)e^{-A} ) \,\,\, \textit{as $n \to \infty$}   
\end{align*}
and
\begin{align*}
&\int_A^{a_n}  \frac{y e^{w_n^{N/(N-1)}(y) - y} g( N^\frac{1-N}{N} \omega_{N-1}^{-\frac{1}{N}} w_n(y) )  }{ ( 1+N^\frac{1-N}{N} \omega_{N-1}^{-\frac{1}{N}} w_n(y) )^\frac{N}{2(N-1)} } dy \\ &\le \|g\|_{L^\infty} \int_{A}^{a_n}y e^{-3 \log y}\,dy= \|g\|_{L^\infty} \int_{A}^{a_n}\frac{1}{y^2} \,dy \\
&\;\to\; \|g\|_{L^\infty} \int_{A}^{\infty}\frac{1}{y^2} \,dy = \frac{\|g\|_{L^\infty}}{A}\qquad \text{as $n \to \infty$.} 
\end{align*}
Consequently,
$$
\limsup_{n \to \infty} \int_0^{a_n}  \frac{y e^{w_n^{N/(N-1)}(y) - y} g( N^\frac{1-N}{N} \omega_{N-1}^{-\frac{1}{N}} w_n(y) )  }{ ( 1+N^\frac{1-N}{N} \omega_{N-1}^{-\frac{1}{N}} w_n(y) )^\frac{N}{2(N-1)} } dy
\le  g(0) ( 1 - (1+A)e^{-A} ) + \frac{\|g\|_{L^\infty}}{A},
$$
for all $A>0$, which gives \eqref{limit-g-0}.

Next, we split the integral in \eqref{upper_bound} into three parts, i.e.,
\begin{align}
&\int_0^{+\infty} \int_y^{+\infty}  \frac{y e^{w_n^{N/(N-1)}(y) - y} g( N^\frac{1-N}{N} \omega_{N-1}^{-\frac{1}{N}} w_n(y) )  }{ ( 1+N^\frac{1-N}{N} \omega_{N-1}^{-\frac{1}{N}} w_n(y) )^\frac{N}{2(N-1)} } \frac{ e^{ w_n^{N/(N-1)} (x) - x} g( N^\frac{1-N}{N} \omega_{N-1}^{-\frac{1}{N}} w_n(x) ) }{( 1+N^\frac{1-N}{N} \omega_{N-1}^{-\frac{1}{N}} w_n(x) )^\frac{N}{2(N-1)}} dx dy  \nonumber\\
&= \int_{0}^{a_n} \int_{a_n}^{+\infty}\dots dx dy + \int_{0}^{a_n} \int_{y}^{a_n} \dots dx dy \nonumber+ \int_{a_n}^{+\infty} \int_{y}^{+\infty} \dots dx dy \nonumber\\
&=: I_n + J_n + K_n \label{int-splitting}.
\end{align}
By (\ref{eq:cc-estimate}) and \eqref{limit-g-0} we have,
\begin{align}
I_n &= \left( \int_{0}^{a_n} \frac{y e^{w_n^{N/(N-1)}(y) - y} g( N^\frac{1-N}{N} \omega_{N-1}^{-\frac{1}{N}} w_n(y) )  }{ ( 1+N^\frac{1-N}{N} \omega_{N-1}^{-\frac{1}{N}} w_n(y) )^\frac{N}{2(N-1)} } dy \right)\nonumber\\ & \times \left( \int_{a_n}^{+\infty} \frac{ e^{ w_n^{N/(N-1)} (x) - x} g( N^\frac{1-N}{N} \omega_{N-1}^{-\frac{1}{N}} w_n(x) ) }{( 1+N^\frac{1-N}{N} \omega_{N-1}^{-\frac{1}{N}} w_n(x) )^\frac{N}{2(N-1)}} dx \right) \nonumber\\
&\leq \frac{\|g\|_{L^\infty}}{( 1+N^\frac{1-N}{N} \omega_{N-1}^{-\frac{1}{N}} w_n(a_n) )^\frac{N}{2(N-1)}} \left( \int_{0}^{a_n} \frac{y e^{w_n^{N/(N-1)}(y) - y} g( N^\frac{1-N}{N} \omega_{N-1}^{-\frac{1}{N}} w_n(y) )  }{ ( 1+N^\frac{1-N}{N} \omega_{N-1}^{-\frac{1}{N}} w_n(y) )^\frac{N}{2(N-1)} } dy \right)\nonumber \\ &  \times \left( \int_{a_n}^{+\infty} e^{w_n^{N/(N-1)}(x) - x}dx \right)  \to 0 \quad \text{as}~ n \to +\infty. \label{I-n-upper-bound}
\end{align}

To estimate $J_n$, we follow the idea of proving \eqref{limit-g-0}. We have 
\begin{align*}
  J_n &\leq \int_0^{a_n} \int_y^{a_n}  y e^{w_n^{N/(N-1)}(y) - y} g( N^\frac{1-N}{N} \omega_{N-1}^{-\frac{1}{N}} w_n(y) )   e^{ w_n^{N/(N-1)} (x) - x} g( N^\frac{1-N}{N} \omega_{N-1}^{-\frac{1}{N}} w_n(x) )   dx dy  \\
  &\leq \int_0^{A} \int_y^{A}\dots dx dy \:+\:  \int_0^{A} \int_A^{a_n} \dots  dx dy \:+\:  \int_{A}^{a_n} \int_y^{a_n}\dots  dx dy, 
\end{align*}
where, as $n \to \infty$, by (\ref{eq:local-uniform-conv}), 
\begin{align}
&  \int_0^{A} y e^{w_n^{N/(N-1)}(y) - y} g( N^\frac{1-N}{N} \omega_{N-1}^{-\frac{1}{N}} w_n(y) )  \int_y^{A} e^{w_n^{N/(N-1)}(x) - x} g( N^\frac{1-N}{N} \omega_{N-1}^{-\frac{1}{N}} w_n(x) ) dx dy  \nonumber\\
  &\to\quad {g^2(0)} \int_0^{A}y e^{- y}  \int_y^{A} e^{- x}  dx dy \le  {g^2(0)} \int_0^{A}y e^{- y}  \int_y^{\infty} e^{- x}  dx dy=
    {g^2(0)} \int_0^{A}y e^{- 2y} dy \nonumber \\
    &\qquad \;= \frac{{g^2(0)}}{4}\int_0^{2A}y e^{- y} dy
  = \frac{{g^2(0)}}{4}\Bigl(1 - (1+2A)e^{-2A}\Bigr), \label{J-n-est-1}\\
&\int_0^{A}y e^{w_n^{N/(N-1)}(y) - y} g( N^\frac{1-N}{N} \omega_{N-1}^{-\frac{1}{N}} w_n(y) )   \int_A^{a_n} e^{w_n^{N/(N-1)}(x) - x} g( N^\frac{1-N}{N} \omega_{N-1}^{-\frac{1}{N}} w_n(x) )  dx dy  \nonumber\\
&\le \|g\|_{L^\infty}^2 \int_0^{A}y e^{w_n^{N/(N-1)}(y) - y}  \int_A^{\infty} \frac{1}{x^3} dx dy = \frac{\|g\|_{L^\infty}^2}{2A^2} \int_0^{A}y e^{w_n^{N/(N-1)}(y) - y} dy \nonumber \\
  &\to \quad \frac{\|g\|_{L^\infty}^2}{2A^2} \int_0^{A}y e^{- y} dy =\frac{\|g\|_{L^\infty}^2}{2A^2} \Bigl(1 - (1+A)e^{-A}\Bigr),\label{J-n-est-2}
\end{align}
and
\begin{align}
&\int_{A}^{a_n}y e^{w_n^{N/(N-1)}(y) - y}  g( N^\frac{1-N}{N} \omega_{N-1}^{-\frac{1}{N}} w_n(y) )  \int_y^{a_n} e^{w_n^{N/(N-1)}(x) - x}  g( N^\frac{1-N}{N} \omega_{N-1}^{-\frac{1}{N}} w_n(x) )  dx dy \nonumber \\
& \le \|g\|_{L^\infty}^2 \int_{A}^{a_n} \frac{1}{y^2} \int_y^{a_n} \frac{1}{x^3}  dx dy \le \|g\|_{L^\infty}^2 \int_{A}^{\infty} \frac{1}{y^2} \int_A^{\infty} \frac{1}{x^3}  dx dy = \frac{\|g\|_{L^\infty}^2}{2A^3} \qquad \text{for all $n \in \N$.} \label{J-n-est-3}
\end{align}
Consequently, by sending $A \to \infty$ in (\ref{J-n-est-1}), (\ref{J-n-est-2}) and (\ref{J-n-est-3}), we deduce that
\begin{equation}
\label{J-n-est}
\limsup_{n \to \infty} J_n \le \frac{{g^2(0)}}{4}.
\end{equation}
To estimate $K_n$, we fix $ \varepsilon \in ( 0 , 1 ) $ and define, for any $ n \in \mathbb{N} $,
\begin{equation*}
H_n^{+} = \{ x \in ( a_n , +\infty ) : w_n(x) \geq  \varepsilon x^\frac{N-1}{N} \}\qquad \text{and}\qquad
H_n^{-} = \{ x \in ( a_n , +\infty ) : w_n(x) <  \varepsilon x^\frac{N-1}{N} \}.
\end{equation*}
Then, for $n$ sufficiently large we have
\begin{align}
&\int_{ [ a_n , +\infty ) \cap H_n^{+}} \int_{ [ y , +\infty ) \cap H_n^{+}} \frac{y e^{w_n^{N/(N-1)}(y) - y} g( N^\frac{1-N}{N} \omega_{N-1}^{-\frac{1}{N}} w_n(y) )  }{ ( 1+N^\frac{1-N}{N} \omega_{N-1}^{-\frac{1}{N}} w_n(y) )^\frac{N}{2(N-1)} }\nonumber  \\ & \nonumber \times \frac{e^{ w_n^{N/(N-1)} (x) - x} g( N^\frac{1-N}{N} \omega_{N-1}^{-\frac{1}{N}} w_n(x) ) }{ ( 1+N^\frac{1-N}{N} \omega_{N-1}^{-\frac{1}{N}} w_n(x) )^\frac{N}{2(N-1)} } dx dy \nonumber  \\
&\footnotesize \leq \int_{ [ a_n , +\infty ) \cap H_n^{+}} \int_{ [ y , +\infty ) \cap H_n^{+}} \frac{y g( N^\frac{1-N}{N} \omega_{N-1}^{-\frac{1}{N}} w_n(y) )  g( N^\frac{1-N}{N} \omega_{N-1}^{-\frac{1}{N}} w_n(x) )  }{ \left( 1 + \varepsilon N^\frac{1-N}{N} \omega_{N-1}^{-\frac{1}{N}}  y^\frac{N-1}{N} \right)^\frac{N}{N-1} }\nonumber \\ & \times  e^{ w_n^{N/(N-1)}(x) + w_n^{N/(N-1)}(y) - (x+y)} dxdy \nonumber  \\
&\leq \frac{N \omega_{N-1}^\frac{1}{N-1}}{\varepsilon^{N/(N-1)}} \int_{a_n}^\infty \int_{y}^\infty e^{ w_n^{N/(N-1)}(x) + w_n^{N/(N-1)}(y) - (x+y)}\nonumber \times \\ & \times  g( N^\frac{1-N}{N} \omega_{N-1}^{-\frac{1}{N}} w_n(y) )  g( N^\frac{1-N}{N} \omega_{N-1}^{-\frac{1}{N}} w_n(x) ) dxdy \nonumber \\
&\nonumber = \frac{N \omega_{N-1}^\frac{1}{N-1}}{2\varepsilon^{N/(N-1)}} \left( \int_{a_n}^\infty e^{ w_n^\frac{N}{N-1}(x) - x} g( N^\frac{1-N}{N} \omega_{N-1}^{-\frac{1}{N}} w_n(x) ) dx \right)^2 \\ &\le \frac{N \omega_{N-1}^\frac{1}{N-1} e^{2(1+\frac{1}{2}+\cdots+\frac{1}{N}) } }{2\varepsilon^{N/(N-1)}} \left( {C_g} + o(1) \right)^2. \label{K-n-upper-bound-0-0-1}
\end{align}
The equality in the last line follows from the symmetry of the integrand and Fubini's theorem, and in the last inequality we used  (\ref{eq:cc-estimate}) together with the fact that
$$
w_n(x) \ge w(a_n) = (a_n - 3 \log a_n)^\frac{N-1}{N} \qquad \text{on $[a_n,\infty)$,}
$$
while $(a_n - 3 \log a_n)^\frac{N-1}{N} \to \infty$ as $a_n \to \infty$. Moreover, for $n$ sufficiently large,
\begin{align}
&\nonumber \int_{ [ a_n , +\infty ) \cap H_n^{+} } \int_{ [ y , +\infty ) \cap H_n^{-}} \frac{y e^{w_n^{N/(N-1)}(y) - y} g( N^\frac{1-N}{N} \omega_{N-1}^{-\frac{1}{N}} w_n(y) )  }{ ( 1+N^\frac{1-N}{N} \omega_{N-1}^{-\frac{1}{N}} w_n(y) )^\frac{N}{2(N-1)} }\times \\ & \times \frac{e^{ w_n^{N/(N-1)} (x) - x} g( N^\frac{1-N}{N} \omega_{N-1}^{-\frac{1}{N}} w_n(x) ) }{ ( 1+N^\frac{1-N}{N} \omega_{N-1}^{-\frac{1}{N}} w_n(x) )^\frac{N}{2(N-1)} } dx dy \nonumber \\
&\leq \int_{ [ a_n , +\infty ) \cap H_n^{+} } \int_{ [ y , +\infty ) \cap H_n^{-}} \frac{y e^{w_n^\frac{N}{N-1}(y) - y} e^{(\varepsilon^{N/(N-1)} - 1) x}  }{ ( 1 + \varepsilon N^\frac{1-N}{N} \omega_{N-1}^{-\frac{1}{N}}  y^\frac{N-1}{N} )^\frac{N}{N-1} }\nonumber \times \\ & \times g( N^\frac{1-N}{N} \omega_{N-1}^{-\frac{1}{N}} w_n(x) ) g( N^\frac{1-N}{N} \omega_{N-1}^{-\frac{1}{N}} w_n(y) ) dxdy\nonumber \\
&\leq \left( {C_g} + o(1) \right)^2 \times \frac{ e^{(\varepsilon^{N/(N-1)} - 1) a_n} }{1-\varepsilon^{N/(N-1)}}\nonumber \times \\ & \times \sup_{y \in [ a_n , +\infty ) } \frac{y}{( 1 + \varepsilon N^\frac{1-N}{N} \omega_{N-1}^{-\frac{1}{N}}  y^\frac{N-1}{N} )^\frac{N}{N-1} } \times \int_{a_n}^\infty e^{w_n^\frac{N}{N-1}(y) - y} dy \nonumber\\
&\leq \left( {C_g} + o(1) \right)^2 \times \frac{ N \omega_{N-1}^\frac{1}{N-1} e^{(\varepsilon^\frac{N}{N-1} - 1) a_n} }{ \varepsilon^\frac{N}{N-1} ( 1-\varepsilon^\frac{N}{N-1} )}  \times \int_{a_n}^\infty e^{w_n^\frac{N}{N-1}(y) - y} dy \to 0 \quad \text{as}~ n \to \infty, \label{K-n-upper-bound-2}
\end{align}
while similarly
\begin{align*}
&\limsup_{n \to \infty} \int_{ [ a_n , +\infty ) \cap H_n^{-} } \int_{ [ y , +\infty ) \cap H_n^{+}} \frac{y e^{w_n^{N/(N-1)}(y) - y} g( N^\frac{1-N}{N} \omega_{N-1}^{-\frac{1}{N}} w_n(y) )  }{ ( 1+N^\frac{1-N}{N} \omega_{N-1}^{-\frac{1}{N}} w_n(y) )^\frac{N}{2(N-1)} }\\ & \times  \frac{e^{ w_n^{N/(N-1)} (x) - x} g( N^\frac{1-N}{N} \omega_{N-1}^{-\frac{1}{N}} w_n(x) ) }{ ( 1+N^\frac{1-N}{N} \omega_{N-1}^{-\frac{1}{N}} w_n(x) )^\frac{N}{2(N-1)} } dx dy \leq 0,
\end{align*}
and
\begin{align*}
&\limsup_{n \to \infty} \int_{ [ a_n , +\infty ) \cap H_n^{-} } \int_{ [ y , +\infty ) \cap H_n^{-}}\frac{y e^{w_n^{N/(N-1)}(y) - y} g( N^\frac{1-N}{N} \omega_{N-1}^{-\frac{1}{N}} w_n(y) )  }{ ( 1+N^\frac{1-N}{N} \omega_{N-1}^{-\frac{1}{N}} w_n(y) )^\frac{N}{2(N-1)} }\\ & \times \frac{e^{ w_n^{N/(N-1)} (x) - x} g( N^\frac{1-N}{N} \omega_{N-1}^{-\frac{1}{N}} w_n(x) ) }{ ( 1+N^\frac{1-N}{N} \omega_{N-1}^{-\frac{1}{N}} w_n(x) )^\frac{N}{2(N-1)} } dx dy \leq 0.
\end{align*}
Combining these asymptotic estimates with (\ref{K-n-upper-bound-0-0-1}) and (\ref{K-n-upper-bound-2}), we obtain that
$$
K_n \le \frac{N \omega_{N-1}^\frac{1}{N-1} e^{2(1+\frac{1}{2}+\cdots+
\frac{1}{N-1}}) } {2\varepsilon^{N/(N-1)}} \left( {C_g} + o(1) \right)^2 \qquad \text{as $n \to \infty$.}
$$
Since $\varepsilon$ can chosen arbitrarily close to $1$ in this estimate, we obtain that 
\begin{equation}
 \label{K-n-upper-bound}
\limsup_{n \to \infty} K_n \leq \frac{N}{2}  \omega_{N-1}^\frac{1}{N-1} C_g^2 e^{2(1+\frac{1}{2}+\cdots+\frac{1}{N-1})  } .
\end{equation}
Collecting the asymptotic estimates for $I_n, J_n$ and $K_n$ given in (\ref{I-n-upper-bound}),(\ref{J-n-est}) and (\ref{K-n-upper-bound}), we get \eqref{upper_bound}. This finishes the proof of (\ref{limsup-SCS}).
\end{proof}

As announced,  we will now readily complete the 

\begin{proof}[Proof of Theorem~\ref{sec:introduction-main-thm-critical-sufficient-cond}]
  By assumption, we have
  \begin{equation}
    \label{eq:m-1-F-assumption}
    m_1(F)> \frac{2 \omega_{N-1}^2}{N^3} \Bigl(\frac{g^2(0)}{4} + \frac{N}{2}  \omega_{N-1}^\frac{1}{N-1} C_g^2 e^{2(1+\frac{1}{2}+\cdots+\frac{1}{N-1})  } \Bigr).          
  \end{equation}
  Let $(u_n)_n$ be a maximizing sequence in $\cB_1$ for $m_F$.  By the Polya-Szego inequality, we may assume that $u_n \in \cB_1^*$ for $n \in \N$. Moreover, we may pass to a subsequence with $u_n \weak u \in H^1_0(B_1)$. If $(u_n)_n$ is an $SCS$-sequence, then it follows from Theorem~\ref{sec:introduction-CC-theorem}(i) that
  $$
  m_1(F)= \lim_{n \to \infty}\Phi(u_n) \le \frac{2 \omega_{N-1}^2}{N^3} \Bigl(\frac{g^2(0)}{4} + \frac{N}{2}  \omega_{N-1}^\frac{1}{N-1} C_g^2 e^{2(1+\frac{1}{2}+\cdots+\frac{1}{N-1})}  \Bigr),          
  $$
  contrary to (\ref{eq:m-1-F-assumption}). Hence $(u_n)_n$ is no $SCS$-sequence, which implies that $u \not = 0$ and therefore
  $$
  m_1(F)= \lim_{n \to \infty}\Phi(u_n)= \Phi(u),
  $$
 by Proposition~\ref{abstract-strong-continuity}, showing that $m_1(F)$ is attained at $u \in \cB_1^*$.
\end{proof}

\section{Sharpness of the upper limit and convergence from above}
\label{sec:sharpn-upper-limit}
In this section, we verify that the upper estimate in (\ref{limsup-SCS}) is sharp. Specifically, we address the second assertion of Theorem~\ref{sec:introduction-CC-theorem}: we will show the existence of a Schwarz symmetric concentrating sequence compliant with (\ref{sharp-lim-SCS}), given that
\begin{equation}
  \label{eq:existence-C-g-limit}
 \exists\, C_g = \lim_{t \to \infty}g(t)  \in (0,\infty). 
\end{equation}
Should the condition (\ref{g-2-prime}) also hold, the relation (\ref{sharp-lim-SCS-above}) is obtained.

As in the last section, we will consider the transformation $r= e^{-t/N}$ of the radial variable. By Lemma~\ref{transformation}, the proof of (\ref{sharp-lim-SCS}) and (\ref{sharp-lim-SCS-above}) is reduced to the following Proposition.

\begin{proposition}
  \label{prop-transformed-sharp-lim}
Let $g$ satisfy $(g_0)$, $(g_1)$ and (\ref{eq:existence-C-g-limit}). There exists a sequence of increasing functions $w_n \in W^{1,N}_{loc}(\R_+)$ with $w_n(0) = 0$ for all $n \in \N$ and the following properties:
  \begin{itemize}
  \item[(i)] 
  \begin{equation}
    \label{eq:local-zero-deriv-conv-0-sharp}
\int_0^\infty (w^\prime_n)^N dt = 1  \quad \text{for all $n \in \N$,}\quad \text{and}\quad \int_0^A w^\prime_n dt \to 0 \quad \text{as $n \to \infty$ for every $A>0$.}
\end{equation}
\item[(ii)]         \begin{equation}
          \label{sharp-lim-SCS-transformed}
          \begin{split}
\lim_{n \to \infty}\int_0^\infty \frac{y e^{w_n^{N/(N-1)}(y) - y} g( N^\frac{1-N}{N} \omega_{N-1}^{-\frac{1}{N}} w_n(y) )  }{ ( 1+N^\frac{1-N}{N} \omega_{N-1}^{-\frac{1}{N}} w_n(y) )^\frac{N}{2(N-1)} } \int_{y}^\infty \frac{ e^{ w_n^{N/(N-1)} (x) - x} g( N^\frac{1-N}{N} \omega_{N-1}^{-\frac{1}{N}} w_n(x) ) }{( 1+N^\frac{1-N}{N} \omega_{N-1}^{-\frac{1}{N}} w_n(x) )^\frac{N}{2(N-1)}} dx dy \\
= \frac{{g^2(0)}}{4} + \frac{N}{2}  \omega_{N-1}^\frac{1}{N-1} C_g^2 e^{2(1+\frac{1}{2}+\cdots+\frac{1}{N-1})}).
\end{split}
\end{equation}
\item[(iii)] If in addition (\ref{g-2-prime}) holds, then 
        \begin{equation}
          \label{sharp-lim-SCS-above-transformed}
\begin{split}
\int_0^\infty \frac{y e^{w_n^{N/(N-1)}(y) - y} g( N^\frac{1-N}{N} \omega_{N-1}^{-\frac{1}{N}} w_n(y) )  }{ ( 1+N^\frac{1-N}{N} \omega_{N-1}^{-\frac{1}{N}} w_n(y) )^\frac{N}{2(N-1)} } \int_{y}^\infty \frac{ e^{ w_n^{N/(N-1)} (x) - x} g( N^\frac{1-N}{N} \omega_{N-1}^{-\frac{1}{N}} w_n(x) ) }{( 1+N^\frac{1-N}{N} \omega_{N-1}^{-\frac{1}{N}} w_n(x) )^\frac{N}{2(N-1)}} dx dy \\
> \frac{{g^2(0)}}{4} + \frac{N}{2}  \omega_{N-1}^\frac{1}{N-1} C_g^2 e^{2(1+\frac{1}{2}+\cdots+\frac{1}{N-1}}).
\end{split}
\end{equation}  
for large $n$.
\end{itemize}
\end{proposition}

The remainder of this section is devoted to the proof of Proposition~\ref{prop-transformed-sharp-lim}. We fix $s  \in (0,\frac{1}{N})$ in the following. Moreover, for $n \in \mathbb{N}$, we set $ \delta_n = \frac{s \log n}{n}$ and define
$w_n \in W^{1,N}_{loc}(\R_+)$ by 
\begin{equation}
  \label{def-w-n}
w_n(t) = \left\{
\begin{aligned}
&\frac{t}{n^{1/N}} \left( 1 - \delta_n \right)^{(N-1)/N},&&\qquad 0 \leq t \leq n, \\
&\frac{N-1}{ ( n (1-\delta_n) )^{1/N}} \log \frac{ A_n + 1 }{ A_n + e^{-(t-n)/(N-1)} } + (n(1-\delta_n))^{(N-1)/N},&&\qquad t \ge n.
\end{aligned}
\right.
\end{equation}
Here we choose $A_n>0$ such that
\begin{equation}
\int_0^\infty |w_n'(t)|^N \,dt = 1,\qquad \text{i.e.}\quad \int_{n}^\infty |w_n'(t)|^N\,dt = 1 - (1-\delta_n)^{N-1}.
  \label{eq:n-infty-cond}
\end{equation}
To see that such a value $A_n$ exists, we note that (\ref{eq:n-infty-cond}) is equivalent to
$$
\frac{N-1}{n(1-\delta_n)}\Bigl(\log \frac{A_n+1}{A_n} - \sum_{k=1}^{N-1} \frac{1}{(N-k) (A_n +1)^{N-k} }\Bigr) = 1 - (1-\delta_n)^{N-1},
$$
i.e.,
\begin{equation}
  \label{eq:defining-eq-An}
\frac{A_n+1}{A_n} = e^{\sum_{k=1}^{N-1} \frac{1}{(N-k) (A_n +1)^{N-k} }} n^{-\frac{s}{N-1}} e^{\frac{n}{N-1} [1-(1-\delta_n)^N]}.
\end{equation}
So $A_n$ is chosen as an intersection point of the functions
$$
x \mapsto f(x):= \frac{x+1}{x}\quad \text{and}\quad x \mapsto h_n(x):= e^{\sum_{k=1}^{N-1} \frac{1}{(N-k) (x +1)^{N-k} }} n^{-\frac{s}{N-1}} e^{\frac{n}{N-1} [1-(1-\delta_n)^N]},
$$
which exists since
$$
\lim_{x \to 0^+}f(x)=+\infty > e^{\sum_{k=1}^{N-1} \frac{1}{N-k}}   n^{-\frac{s}{N-1}} e^{\frac{n}{N-1} [1-(1-\delta_n)^N]} = h_n(0)
$$
and
$$
\lim_{x \to \infty}f(x) = 1 <    n^{-\frac{s}{N-1}} e^{\frac{n}{N-1} [1-(1-\delta_n)^N]} =:h_{n,\infty} = \lim_{x \to \infty}h_n(x).
$$
This choice of $w_n$ is inspired by Figueiredo et al. \cite{figuereido-do-o-ruf}, who considered the special case $s=2$ in the definition of $w_n$. For our purposes, it turns out that the complementary choice of $s \in (0,\frac{1}{N})$ is essential. 

Partly following \cite{figuereido-do-o-ruf}, we first note some elementary estimates. 
Since $h_n(x) \ge h_{n,\infty} \to \infty$ as $n \to \infty$ for $x>0$, it follows that $A_n \to 0$. More precisely, we have

\begin{lemma}
  \label{precise-A-n-asymptotics}
As $n \to \infty$, we have   
\begin{equation}
  \label{eq:A-n-precise-est}
A_n = n^{-s} \frac{1}{e^{1+\frac{1}{2}+\cdots+\frac{1}{N-1}} +O(n^{-s})}= \frac{1}{e^{1+\frac{1}{2}+\cdots+\frac{1}{N-1} }} n^{-s} +O(n^{-2s})
\end{equation}
and
\begin{equation}
\label{Ruf-Remainder-est}
\int_{n}^\infty e^{w_n^{N/(N-1)}(t)-t}\,dt \ge e^{1+\frac{1}{2}+\cdots+\frac{1}{N-1}} - n^{-s} + o(n^{-s}). 
\end{equation}
\end{lemma}

\begin{proof}
As $n \to \infty$, we have  
\begin{equation}
  \label{eq:intermediate-1-section-sharp}
n^{-\frac{s}{N-1}} e^{\frac{n}{N-1} [1-(1-\delta_n)^N]} = n^{-\frac{s}{N-1}} e^{\frac{n}{N-1} ( N \delta_n + O(\delta_n^2))} = n^s e^{O(n\delta_n^2)}.
\end{equation}
Inserting this expansion in ~(\ref{eq:defining-eq-An}) gives 
$A_n = O(n^{-s})$ and
\begin{align}
  \label{eq:intermediate-2-section-sharp}
&\frac{A_n+1}{A_n n^s} =e^{\sum_{k=1}^{N-1} \frac{1}{(N-k) (A_n +1)^{N-k} }} e^{O(n\delta_n^2)}= e^{\sum_{k=1}^{N-1} \frac{1}{N-k}} (1+O(n^{-s}))(1+O(\frac{\log^2 n}{n}))\nonumber \\ &=e^{\sum_{k=1}^{N-1} \frac{1}{N-k}} +O(n^{-s}).
\end{align}
as $n \to \infty$. We point out that here it is essential that $s<1$. Consequently, 
$$
1+ \frac{1}{A_n}= \frac{A_n+1}{A_n} = n^s( e^{\sum_{k=1}^{N-1} \frac{1}{N-k}} + O(n^{-s}))
$$
and hence (\ref{eq:A-n-precise-est}) follows.\\
To see (\ref{Ruf-Remainder-est}), we follow the calculations in \cite[P. 147]{figuereido-do-o-ruf}, replacing $s=2$ by $s \in (0,\frac{1}{N})$ to see that
$$
  \int_{n}^\infty e^{w_n^{N/(N-1)}(t)-t}\,dt \ge \frac{A_n+1}{A_n n^s},
$$
which by (\ref{eq:intermediate-1-section-sharp}) and (\ref{eq:intermediate-2-section-sharp}) gives 
  \begin{align*}
    \int_{n}^\infty e^{w_n^{N/(N-1)}(t)-t}\,dt &\ge  e^{1+\frac{1}{2}+\cdots+\frac{1}{N-1}} ( 1- (N-1) A_n)(1+ O(\frac{\log^2 n}{n}))\\
 &=e^{1+\frac{1}{2}+\cdots+\frac{1}{N-1}} (1-\frac{ (N-1) n^{-s}}{e^{1+\frac{1}{2}+\cdots+\frac{1}{N-1}}} +  O(n^{-2s}))(1+ O(\frac{\log^2 n}{n})) \\
 &\ge e^{1+\frac{1}{2}+\cdots+\frac{1}{N-1}}- n^{-s} + o(n^{-s}).
\end{align*}
\end{proof}
We also need another technical lemma.
 
\begin{lemma}\label{xi-choice}
For any $ n $, there exists a unique $\xi_n = \xi(n,N) \in (0,1) $ such that \begin{align}\label{technical relationship}
    \int_0^{\xi_n n} e^{n^{-1/(N-1)} y (y^\frac{1}{N-1} - n^\frac{1}{N-1}) } \,dy = \int_{\xi_n n}^n e^{n^{-1/(N-1)} y (y^\frac{1}{N-1} - n^\frac{1}{N-1}) } \,dy,
\end{align} and moreover, $\xi_n \geq (\frac{N-1}{N})^{N-1}.$
\end{lemma}
\begin{proof}
   By applying the change of variable, the condition for $\xi_n \in (0,1)$ in \eqref{technical relationship} can be equivalently rewritten as 
    $$
    \int_0^{\xi_n} e^{n y (y^\frac{1}{N-1} - 1) } \,dy = \int_{\xi_n}^1 e^{n y (y^\frac{1}{N-1} - 1) } \,dy.
    $$
    To show the existence and uniqueness of $\xi_n$, let $$h_n(x) := \int_0^{x} e^{n y (y^\frac{1}{N-1} - 1) } dy- \int_{x}^1 e^{n y (y^\frac{1}{N-1} - 1) } \,dy,\,\,\,  x \in [0,1].$$ For any fixed $n$, $h_n(x)$ is continuous, strictly increasing and $h_n(0) = -h_n(1) < 0$, then there exists unique $\xi_n \in (0,1)$ such that $h_n(\xi_n) = 0$.\\ \noindent  To show the lower bound of $\xi_n$, we consider the function \begin{align*}
        \varphi_n(y) &:= {n (y^\frac{N}{N-1} - y)} - {\frac{N^{N-2}}{ N^{N-1} - (N-1)^{N-1} }  n (y^2-y)} \\ &= n y \bigl( - \frac{N^{N-2}}{ N^{N-1} - (N-1)^{N-1} } y + y^\frac{1}{N-1} \\ &+ (\frac{N^{N-2}}{ N^{N-1} - (N-1)^{N-1}} -1) \bigr) =: n y \psi(y),\,\,\,  y \in [0,1].
    \end{align*}   We note by definition of $\varphi_n$ that $\varphi_n(0) = \varphi_n(1) = \varphi_n( (\frac{N-1}{N})^{N-1} ) = 0$. Moreover, $\psi(y)$ is increasing in $(0,\alpha_N)$ and decreasing in $(\alpha_N,1)$, where $$\alpha_N := (N-1)^\frac{N-1}{2-N} \Big(\frac{N^{N-2}}{ N^{N-1} - (N-1)^{N-1}}\Big)^\frac{N-1}{2-N} \in (0,1) $$ such that $\psi^\prime (\alpha_N) = 0$. Then, we have  $$\varphi_n(y) < 0 \,\,\, \text{in}\,\,\, (0,(\frac{N-1}{N})^{N-1}),$$   while $$\varphi_n(y) > 0\,\,\, \text{in}\,\,\, ((\frac{N-1}{N})^{N-1},1).$$ This implies, using the monotonicity of the exponential function, that \begin{align}\label{phi negative}
        e^{n (y^\frac{N}{N-1} - y)} < e^{\frac{N^{N-2}}{ N^{N-1} - (N-1)^{N-1} }  n (y^2-y)}\,\,\, \text{in}\,\,\, (0,(\frac{N-1}{N})^{N-1})
    \end{align} and \begin{align}\label{phi positive}
        e^{n (y^\frac{N}{N-1} - y)} > e^{\frac{N^{N-2}}{ N^{N-1} - (N-1)^{N-1} }  n (y^2-y)}\,\,\, \text{in}\,\,\, ((\frac{N-1}{N})^{N-1},1).
    \end{align} Since $(\frac{N-1}{N})^{N-1} < \frac{1}{2}$, then by \eqref{phi negative}, \eqref{phi positive} and symmetry of quadratic function, we have 
    $$
    \begin{aligned}
   & h_n\left( \left( \frac{N-1}{N} \right)^{N-1} \right) = \int_0^{\left( \frac{N-1}{N} \right)^{N-1}} e^{n y (y^\frac{1}{N-1} - 1) } dy - \int_{\left( \frac{N-1}{N} \right)^{N-1}}^1 e^{n y (y^\frac{1}{N-1} - 1) } dy \\
    &\leq \int_0^{\left( \frac{N-1}{N} \right)^{N-1}}  e^{\frac{N^{N-2}}{ N^{N-1} - (N-1)^{N-1} }  n (y^2-y)} dy - \int_{\left( \frac{N-1}{N} \right)^{N-1}}^1 e^{\frac{N^{N-2}}{ N^{N-1} - (N-1)^{N-1} }  n (y^2-y)} dy \\
    &\leq \int_0^\frac{1}{2} e^{\frac{N^{N-2}}{ N^{N-1} - (N-1)^{N-1} }  n (y^2-y)} dy - \int_\frac{1}{2}^1 e^{\frac{N^{N-2}}{ N^{N-1} - (N-1)^{N-1} }  n (y^2-y)} dy = 0.
    \end{aligned}
    $$
As a consequence, $\xi_n \geq (\frac{N-1}{N})^{N-1}$ for all $ n \in \mathbb{N}$.
\end{proof}

Next, assuming that $(g_0)$, $(g_1)$ and (\ref{eq:existence-C-g-limit}) holds, we consider a similar decomposition as in (\ref{int-splitting}) for the double integral in (\ref{sharp-lim-SCS-transformed}), with $a_n$ replaced by $n$. First, we have
$$
\begin{aligned}
  \tilde{I}_n &:= \int_{0}^{n} \int_{n}^{+\infty} \frac{y e^{w_n^{N/(N-1)}(y) - y} g( N^\frac{1-N}{N} \omega_{N-1}^{-\frac{1}{N}} w_n(y) )  }{ ( 1+N^\frac{1-N}{N} \omega_{N-1}^{-\frac{1}{N}} w_n(y) )^\frac{N}{2(N-1)} } \frac{ e^{ w_n^{N/(N-1)} (x) - x} g( N^\frac{1-N}{N} \omega_{N-1}^{-\frac{1}{N}} w_n(x) ) }{( 1+N^\frac{1-N}{N} \omega_{N-1}^{-\frac{1}{N}} w_n(x) )^\frac{N}{2(N-1)}} dx dy \\
&= \int_{0}^{n} \frac{y e^{w_n^{N/(N-1)}(y) - y} g( N^\frac{1-N}{N} \omega_{N-1}^{-\frac{1}{N}} w_n(y) )  }{ ( 1+N^\frac{1-N}{N} \omega_{N-1}^{-\frac{1}{N}} w_n(y) )^\frac{N}{2(N-1)} } \,dy   \int_n^\infty \frac{ e^{ w_n^{N/(N-1)} (x) - x} g( N^\frac{1-N}{N} \omega_{N-1}^{-\frac{1}{N}} w_n(x) ) }{( 1+N^\frac{1-N}{N} \omega_{N-1}^{-\frac{1}{N}} w_n(x) )^\frac{N}{2(N-1)}} dx \\
&= \tilde{I}_n^1  \times \tilde{I}_n^2,
\end{aligned}
$$
where, by (\ref{eq:existence-C-g-limit}) and choosing $\xi_n$ as Lemma \ref{xi-choice},
\begin{align*}
  &\tilde{I}_n^1 \ge   \int_{\xi_n n}^{n} \frac{y e^{w_n^{N/(N-1)}(y) - y} g( N^\frac{1-N}{N} \omega_{N-1}^{-\frac{1}{N}} w_n(y) )  }{ ( 1+N^\frac{1-N}{N} \omega_{N-1}^{-\frac{1}{N}} w_n(y) )^\frac{N}{2(N-1)} } \,dy \\
  & \ge \Bigl( \inf_{\xi_n n \le y \le n}\frac{y}{( 1+N^\frac{1-N}{N} \omega_{N-1}^{-\frac{1}{N}} w_n(y) )^\frac{N}{2(N-1)}}\Bigr) \left( {C_g} + o(1) \right) \int_{\xi_n n}^{n} e^{w_n^{N/(N-1)}(y)-y}\,dy\\
  &= \Bigl(\inf_{\xi_n n \le y \le n} \frac{y}{( 1+N^\frac{1-N}{N} \omega_{N-1}^{-\frac{1}{N}} w_n(y) )^\frac{N}{2(N-1)}}\Bigr)  \left( {C_g} + o(1) \right) \int_{\xi_n n}^{n} e^{\frac{1-\delta_n}{n^{1/(N-1)}} y^\frac{N}{N-1}-y}\,dy\\
  &= \frac{y}{( 1+N^\frac{1-N}{N} \omega_{N-1}^{-\frac{1}{N}} \frac{y}{n^{1/N}} (1-\delta_n)^{(N-1)/N} )^\frac{N}{2(N-1)}}|_{y=\xi_n n}  \left( {C_g} + o(1) \right) \int_{\xi_n n}^{n} e^{\frac{1-\delta_n}{n^{1/(N-1)}} y^\frac{N}{N-1}-y} \,dy\\
  &= \frac{\xi_n n}{ (1+ N^\frac{1-N}{N} \omega_{N-1}^{-\frac{1}{N}} \xi_n   (n(1-\delta_n))^{(N-1)/N} )^\frac{N}{2(N-1)}}  \left( {C_g} + o(1) \right) \\ & \times \int_{\xi_n n}^{n} e^{n^{-\frac{1}{N-1}} y (y^\frac{1}{N-1} - n^\frac{1}{N-1})} e^{-\frac{\delta_n}{n^{1/(N-1)}} y^{N/(N-1)}} \,dy \\
  &\ge \frac{\xi_n n e^{-n \delta_n}}{ (1+ N^\frac{1-N}{N} \omega_{N-1}^{-\frac{1}{N}} \xi_n   (n(1-\delta_n))^{(N-1)/N} )^\frac{N}{2(N-1)}}  \left( {C_g} + o(1) \right) \int_{\xi_n n}^{n} e^{n^{-\frac{1}{N-1}} y (y^\frac{1}{N-1} - n^\frac{1}{N-1})} \,dy \\
  &= \frac{ \xi_n n^{1-s}}{(1+ N^\frac{1-N}{N} \omega_{N-1}^{-\frac{1}{N}} \xi_n   (n(1-\delta_n))^{(N-1)/N} )^\frac{N}{2(N-1)}}  \left( {C_g} + o(1) \right) \int_{0}^{\xi_n n} e^{n^{-\frac{1}{N-1}} y (y^\frac{1}{N-1} - n^\frac{1}{N-1})} \,dy \\
  &\ge  \frac{ \xi_n n^{1-s}}{(1+ N^\frac{1-N}{N} \omega_{N-1}^{-\frac{1}{N}} \xi_n   (n(1-\delta_n))^{(N-1)/N} )^\frac{N}{2(N-1)}}  \left( {C_g} + o(1) \right) \int_{0}^{\xi_n n} e^{- y}\,dy \\
  &= \Bigl(\xi_n^\frac{N-2}{2(N-1)} N^\frac{1}{2} \omega_{N-1}^\frac{1}{2(N-1)} n^{1/2-s} +o(n^{1/2-s}) \Bigr)  \left(  C_g + o(1) \right)(1+o(1)) \\
  &= {C_g} \left({\xi_n}^\frac{N-2}{N-1} N \omega_{N-1}^\frac{1}{N-1}\right)^{1/2} n^{1/2-s} +o(n^{1/2-s}) \\
  &\geq C_g \left( (\frac{N-1}{N})^{N-2} N \omega_{N-1}^\frac{1}{N-1} \right)^\frac{1}{2} n^{1/2-s} +o(n^{1/2-s})
\end{align*} 
and, by (\ref{Ruf-Remainder-est}),
\begin{align*}
 \tilde{I}_n^2& \ge  \frac{ {C_g} + o(1) }{[1 + N^{(1-N)/N} \omega_{N-1}^{-1/N} \Bigl(  \frac{N-1}{ \left( n \left( 1 - \delta_n \right)^{1/N} \right)} \log \frac{
A_n + 1}{A_n} + \left( n \left( 1 - \delta_n \right) \right)^{(N-1)/N}\Bigr) ]^{N/(2(N-1))} }\\ & \times \int_n^{+\infty} e^{ w_n^{N/(N-1)}(x) - x} dx \\
              &= \Bigl( N^\frac{1}{2} \omega_{N-1}^\frac{1}{2(N-1)} n^{-\frac{1}{2}} + o(n^{-\frac{1}{2}})\Bigr) \Bigl( {C_g} + o(1) \Bigr) \Bigl( e^{1+\frac{1}{2}+\cdots+\frac{1}{N-1}}+O(n^{-s})\Bigr) \\
              &= {C_g} N^\frac{1}{2} \omega_{N-1}^\frac{1}{2(N-1)} e^{1+\frac{1}{2}+\cdots+\frac{1}{N-1}} n^{-\frac{1}{2}} + o(n^{-\frac{1}{2}}).
\end{align*}
Consequently,
\begin{align}\label{low-bound-I-n}
  \tilde{I}_n  & \nonumber \ge  \Bigl( C_g \left( (\frac{N-1}{N})^{N-2} N \omega_{N-1}^\frac{1}{N-1} \right)^\frac{1}{2} n^{1/2-s} +o(n^{1/2-s}) \Bigr)\\ & \nonumber \times \Bigl( {C_g} N^\frac{1}{2} \omega_{N-1}^\frac{1}{2(N-1)} e^{1+\frac{1}{2}+\cdots+\frac{1}{N-1}} n^{-\frac{1}{2}} + o(n^{-\frac{1}{2}}) \Bigr)  \\
  &\ge  (\frac{N-1}{N})^\frac{N-2}{2} N C_g^2 \omega_{N-1}^\frac{1}{N-1} e^{1+\frac{1}{2}+\cdots+\frac{1}{N-1}} n^{-s} +o(n^{-s}).
\end{align}
Moreover, fixing $\tau \in (0,\frac{1}{N}-s)$, we find that
$$
0 \le N^\frac{1-N}{N} \omega_{N-1}^{-\frac{1}{N}} w_n(x) \le   n^{\tau-\frac{1}{N}}
              \left( 1 - \delta_n \right)^{(N-1)/N} \le t_0 \quad \text{for $x \in [0,n^{\tau}]$ and $n$ sufficiently large,} 
$$
where $t_0$ is chosen as in Remark~\ref{remark-g-0-assumption} so that $g(t) \ge g(0)$ for $0 \le t \le t_0$. 
Consequently, 
\begin{align}
\tilde{J}_n &:= \int_0^n \int_y^n \frac{y e^{w_n^{N/(N-1)}(y) - y} g( N^\frac{1-N}{N} \omega_{N-1}^{-\frac{1}{N}} w_n(y) )  }{ ( 1+N^\frac{1-N}{N} \omega_{N-1}^{-\frac{1}{N}} w_n(y) )^\frac{N}{2(N-1)} } \frac{ e^{ w_n^{N/(N-1)} (x) - x} g( N^\frac{1-N}{N} \omega_{N-1}^{-\frac{1}{N}} w_n(x) ) }{( 1+N^\frac{1-N}{N} \omega_{N-1}^{-\frac{1}{N}} w_n(x) )^\frac{N}{2(N-1)}} dx dy \nonumber \\
&\geq \int_0^{n^\tau} \int_y^{n^\tau} \frac{ y e^{-y} g( N^\frac{1-N}{N} \omega_{N-1}^{-\frac{1}{N}} w_n(y) ) }{ ( 1+N^\frac{1-N}{N} \omega_{N-1}^{-\frac{1}{N}} w_n(y) )^\frac{N}{2(N-1)} } \frac{ e^{-x} g( N^\frac{1-N}{N} \omega_{N-1}^{-\frac{1}{N}} w_n(x) ) }{ ( 1+N^\frac{1-N}{N} \omega_{N-1}^{-\frac{1}{N}} w_n(x) )^\frac{N}{2(N-1)} } dxdy \nonumber \\
&\geq \frac{g^2(0)}{ [ 1 + N^\frac{1-N}{N} \omega_{N-1}^{-\frac{1}{N}} n^{\tau-\frac{1}{N}}
              \left( 1 - \delta_n \right)^{(N-1)/N} ]^\frac{N}{(N-1)} }  \int_0^{{n^\tau}} \int_y^{{n^\tau}} y e^{-x -y} dxdy \nonumber \\
&=  \frac{{g^2(0)}}{[ 1 + N^\frac{1-N}{N} \omega_{N-1}^{-\frac{1}{N}} n^{\tau-\frac{1}{N}}
              \left( 1 - \delta_n \right)^{(N-1)/N} ]^\frac{N}{(N-1)} }  \left( \frac{1}{4} +O(n^{\tau}) e^{-{n^\tau}}\right) \nonumber \\
            &= {g^2(0)} \left(1 + O(n^{\tau-\frac{1}{N}}) \right)  \left( \frac{1}{4} +O(n^{\tau}) e^{-{n^\tau}}\right)= \frac{{g^2(0)}}{4} + O(n^{\tau-\frac{1}{N}}) \qquad \text{as $\;n \to \infty$.}
   \label{eq:tilde-J-n-est}
\end{align}
Finally, we have 
\begin{equation}\label{low-bound-K-n-1}
\begin{aligned}
\tilde{K}_n : &= \int_n^{+\infty} \frac{y e^{w_n^{N/(N-1)}(y) - y} g( N^\frac{1-N}{N} \omega_{N-1}^{-\frac{1}{N}} w_n(y) )  }{ ( 1+N^\frac{1-N}{N} \omega_{N-1}^{-\frac{1}{N}} w_n(y) )^\frac{N}{2(N-1)} } \int_y^{+\infty} \frac{ e^{ w_n^{N/(N-1)} (x) - x} g( N^\frac{1-N}{N} \omega_{N-1}^{-\frac{1}{N}} w_n(x) ) }{( 1+N^\frac{1-N}{N} \omega_{N-1}^{-\frac{1}{N}} w_n(x) )^\frac{N}{2(N-1)}} dx dy \\
&\geq \frac{n}{ \left( 1 + N^\frac{1-N}{N} \omega_{N-1}^{-\frac{1}{N}} \left( \frac{N-1}{ \left( n \left( 1 - \delta_n \right) \right)^{1/N}}  \log \frac{A_n + 1}{ A_n } + \left( n \left( 1 - \delta_n \right) \right)^{(N-1)/N} \right) \right)^{N/(N-1)} } \\
&~~~~\times \int_n^{+\infty} \int_y^{+\infty} e^{ w_n^{N/(N-1)} (x) + w_n^{N/(N-1)} (y) - x - y}  g( N^\frac{1-N}{N} \omega_{N-1}^{-\frac{1}{N}} w_n(x) ) g( N^\frac{1-N}{N} \omega_{N-1}^{-\frac{1}{N}} w_n(y) ) dxdy,
\end{aligned}
\end{equation}
where
\begin{equation}\label{low-bound-K-n-2}
\begin{aligned}
&\frac{N-1}{ \left( n \left( 1 - \delta_n \right) \right)^{1/N}}  \log \frac{A_n + 1}{ A_n } + \left( n \left( 1 - \delta_n \right) \right)^{(N-1)/N}=\frac{(N-1)\log \frac{1}{A_n} + o(1)}{ \left( n \left( 1 - \delta_n \right) \right)^{1/N}} + \left( n \left( 1 - \delta_n \right) \right)^{(N-1)/N}\\
  &=\frac{\log (e n^{s} + o(n^s))    + o(1)}{ \left( n \left( 1 - \delta_n \right) \right)^{1/N}} + \left( n \left( 1 - \delta_n \right) \right)^{(N-1)/N}=\frac{1 + s \log n  + o(1)}{ \left( n \left( 1 - \delta_n \right) \right)^{1/N}} + \left( n \left( 1 - \delta_n \right) \right)^{(N-1)/N}\\
  &=\left( n \left( 1 - \delta_n \right) \right)^{(N-1)/N} +o(1),
\end{aligned}
\end{equation}
and hence
\begin{equation}\label{low-bound-K-n-3}
\begin{aligned}
  &\frac{n}{ \left( 1 + N^\frac{1-N}{N} \omega_{N-1}^{-\frac{1}{N}} \left( \frac{N-1}{ \left( n \left( 1 - \delta_n \right) \right)^{1/N}}  \log \frac{A_n + 1}{ A_n } + \left( n \left( 1 - \delta_n \right) \right)^{(N-1)/N} \right) \right)^{N/(N-1)} }\\
  &=\frac{n}{\left(1 + N^\frac{1-N}{N} \omega_{N-1}^{-\frac{1}{N}} \left ( \left( n \left( 1 - \delta_n \right) \right)^{(N-1)/N} +o(1) \right)\right)^{N/(N-1)}} \\
 &= N \omega^\frac{1}{N-1}_{N-1} + O(n^{-\frac{N-1}{N}}),
\end{aligned}
\end{equation}
which implies that
\begin{equation}\label{low-bound-K-n-4}
\begin{aligned}
\tilde{K}_n &\ge \bigl(N \omega^\frac{1}{N-1}_{N-1} + O(n^{-\frac{N-1}{N}}) \bigr) \times \frac{1}{2} \times \left( \int_n^{+\infty} e^{ w_n^{N/(N-1)} (x) - x } g( N^\frac{1-N}{N} \omega_{N-1}^{-\frac{1}{N}} w_n(x) ) dx \right)^2 \\
&\geq \bigl(\frac{N}{2} \omega^\frac{1}{N-1}_{N-1} +O(n^{-\frac{N-1}{N}}) \bigr) \bigl(e^{1+\frac{1}{2}+\cdots+\frac{1}{N-1}} - n^{-s} + o(n^{-s})\bigr)^2 \left( {C_g} + o(1) \right)^2 \\
  &= \bigl(\frac{N}{2} \omega^\frac{1}{N-1}_{N-1} + O(n^{-\frac{N-1}{N}}) \bigr) \bigl( e^{2(1+\frac{1}{2}+\cdots+\frac{1}{N-1})} - 2e^{1+\frac{1}{2}+\cdots+\frac{1}{N-1}} n^{-s} + o(n^{-s} ) \bigr) \left( {C_g} + o(1) \right)^2 \\
\end{aligned}
\end{equation}
as $n \to \infty$.
Thus,
\begin{equation}
  \label{eq:liminf-tilde-K-n-bound}
\liminf_{n \to +\infty} \tilde{K}_n \geq \frac{N}{2} \omega^\frac{1}{N-1}_{N-1} C_g^2 e^{2(1+\frac{1}{2}+\cdots+\frac{1}{N-1})}.  
\end{equation}
Now, combining estimates \eqref{low-bound-I-n}, (\ref{eq:tilde-J-n-est}) and 
\eqref{eq:liminf-tilde-K-n-bound}, we conclude that
    \begin{equation*}
    \begin{aligned}
\liminf_{n \to \infty} \int_0^\infty \frac{y e^{w_n^{N/(N-1)}(y) - y} g( N^\frac{1-N}{N} \omega_{N-1}^{-\frac{1}{N}} w_n(y) )  }{ ( 1+N^\frac{1-N}{N} \omega_{N-1}^{-\frac{1}{N}} w_n(y) )^\frac{N}{2(N-1)} } \int_{y}^\infty & \frac{ e^{ w_n^{N/(N-1)} (x) - x} g( N^\frac{1-N}{N} \omega_{N-1}^{-\frac{1}{N}} w_n(x) ) }{( 1+N^\frac{1-N}{N} \omega_{N-1}^{-\frac{1}{N}} w_n(x) )^\frac{N}{2(N-1)}} dx dy \\
&\ge  \frac{{g^2(0)}}{4} + \frac{N}{2} \omega^\frac{1}{N-1}_{N-1} C_g^2 e^{2(1+\frac{1}{2}+\cdots+\frac{1}{N-1})}.
\end{aligned}
\end{equation*}
Combining this with (\ref{upper_bound}) which also holds for the sequence $(w_n)_n$ defined in (\ref{def-w-n}), we obtain (\ref{sharp-lim-SCS-transformed}). This finishes the proof of Proposition~\ref{prop-transformed-sharp-lim}(ii).\\

To prove Part (iii) of Proposition~\ref{prop-transformed-sharp-lim}, we now assume (\ref{g-2-prime}), which implies that there exist constants $C_2,t_0>0$ with 
$$
g(t) \ge {C_g} + C_2 t^{-\rho} \qquad \text{for $t \ge t_0$}.
$$
Since for $n$ sufficiently large we have
$$
N^\frac{1-N}{N} \omega_{N-1}^{-\frac{1}{N}} w_n(x)\ge t_0 \qquad \text{for $x \ge n$},
$$
the estimate \eqref{low-bound-K-n-4} can be improved, for $n$ large, as
\begin{align}
\tilde{K}_n &\ge \bigl(N \omega^\frac{1}{N-1}_{N-1} +O(n^{-\frac{N-1}{N}}) \bigr) \times \frac{1}{2} \times \left( \int_n^{+\infty} e^{ w_n^{N/(N-1)} (x) - x } g( N^\frac{1-N}{N} \omega_{N-1}^{-\frac{1}{N}} w_n(x) ) dx \right)^2 \nonumber \\
&\geq \bigl(\frac{N}{2} \omega^\frac{1}{N-1}_{N-1} + O(n^{-\frac{N-1}{N}})\bigr) \bigl(e^{1+\frac{1}{2}+\cdots+\frac{1}{N-1}} - n^{-s} + o(n^{-s})\bigr)^2 \nonumber \\ & \times \left( {C_g} + C_2 (N^{N-1} \omega_{N-1})^\frac{\rho}{N} \inf_{ x \geq n } w_n^{-\rho}(x) \right)^2 \nonumber \\
  &= \nonumber \bigl(\frac{N}{2} \omega^\frac{1}{N-1}_{N-1} + O(n^{-\frac{N-1}{N}}) \bigr) \bigl( e^{2(1+\frac{1}{2}+\cdots+\frac{1}{N-1})} - 2 e^{1+\frac{1}{2}+\cdots+\frac{1}{N-1}} n^{-s} + o(n^{-s}) \bigr) \times \\
  &\times \left( {C_g} + C_2 (N^{N-1} \omega_{N-1})^\frac{\rho}{N} \left( \left( n \left( 1 - \delta_n \right) \right)^\frac{N-1}{N} + o(1) \right)^{-\rho} \right)^2 \nonumber \\
  &=\frac{N}{2} \omega_{N-1}^\frac{1}{N-1} C_g^2 e^{2(1+\frac{1}{2}+\cdots+\frac{1}{N-1})}\nonumber \\ & + N \omega_{N-1}^\frac{1}{N-1} ( N^{N-1} \omega_{N-1} )^\frac{\rho}{N} {C_g} C_2 e^{2(1+\frac{1}{2}+\cdots+\frac{1}{N-1})} n^{-(N-1)\rho/N}+ O(n^{-s}).\label{low-bound-K-n-5}
\end{align}
Combining \eqref{low-bound-I-n}, (\ref{eq:tilde-J-n-est}) and (\ref{low-bound-K-n-5}), we conclude that 
\begin{align}\label{limit-from-above}
  \int_0^\infty & \frac{y e^{w_n^{N/(N-1)}(y) - y} g( N^\frac{1-N}{N} \omega_{N-1}^{-\frac{1}{N}} w_n(y) )  }{ ( 1+N^\frac{1-N}{N} \omega_{N-1}^{-\frac{1}{N}} w_n(y) )^\frac{N}{2(N-1)} } \int_{y}^\infty \frac{ e^{ w_n^{N/(N-1)} (x) - x} g( N^\frac{1-N}{N} \omega_{N-1}^{-\frac{1}{N}} w_n(x) ) }{( 1+N^\frac{1-N}{N} \omega_{N-1}^{-\frac{1}{N}} w_n(x) )^\frac{N}{2(N-1)}} dx dy\\
&  \geq \frac{{g^2(0)}}{4} + O(n^{\tau-\frac{1}{N}}) +  \frac{N}{2} \omega_{N-1}^\frac{1}{N-1} C_g^2 e^{2(1+\frac{1}{2}+\cdots+\frac{1}{N-1})} \\
&+ N \omega_{N-1}^\frac{1}{N-1} ( N^{N-1} \omega_{N-1} )^\frac{\rho}{N} {C_g} C_2 e^{2(1+\frac{1}{2}+\cdots+\frac{1}{N-1})} n^{-(N-1)\rho/N}+ O(n^{-s}) \nonumber\\
&= \frac{{g^2(0)}}{4} +  \frac{N}{2} \omega_{N-1}^\frac{1}{N-1} C_g^2 e^{2(1+\frac{1}{2}+\cdots+\frac{1}{N-1})} \\
&+ N \omega_{N-1}^\frac{1}{N-1} ( N^{N-1} \omega_{N-1} )^\frac{\rho}{N} {C_g} C_2 e^{2(1+\frac{1}{2}+\cdots+\frac{1}{N-1})} n^{-(N-1)\rho/N} + O(n^{-s})\nonumber
\end{align}
as $n \to \infty$. Here we used in the last step that we have chosen $\tau <\frac{1}{N}-s$. Since $\rho<\frac{1}{N-1}$, we can assume that $s
\in (\frac{(N-1) \rho}{N}, \frac{1}{N})$ was chosen here. We then conclude that (\ref{sharp-lim-SCS-above-transformed}) holds for $n$ large, as claimed. This finishes the proof of Proposition~\ref{prop-transformed-sharp-lim}(iii).

\end{document}